\documentclass[a4paper]{article}
  \usepackage{amsmath}
  \usepackage{amssymb}
  \usepackage{setspace}
  \usepackage{amsthm}
  \usepackage{geometry} 
  \usepackage{color}
  \usepackage{lineno}
  
  \newtheorem{thm}{Theorem}[section]
  \newtheorem{defn}[thm]{Definition}
  \newtheorem{prop}[thm]{Proposition}
  
  \newtheorem{lem}[thm]{Lemma}
  \newtheorem{remark}[thm]{Remark}
  \newtheorem*{assumption}{Assumption I}
  
  \usepackage[toc,title,titletoc,header]{appendix} 
  \usepackage{fancyhdr}
  
 \numberwithin{equation}{section}
 
\begin{document}
 \title{\textbf{Edwards-Wilkinson Fluctuations in the Howitt-Warren Flows}}
 
 \author{Jinjiong Yu\footnote{Department of Mathematics, National University of Singapore, 10 Lower Kent Ridge Road, 119076 Singapore.
     E-mail: yujinjiong@nus.edu.sg}}
 \maketitle

 \begin{abstract}
 We study current fluctuations in a one-dimensional interacting particle system known as the dual smoothing process that is dual to random motions in a Howitt-Warren flow. The Howitt-Warren flow can be regarded as the transition kernels of a random motion in a continuous space-time random environment. It turns out that the current fluctuations of the dual smoothing process fall in the Edwards-Wilkinson universality class, where the fluctuations occur on the scale $t^{1/4}$ and the limit is a universal Gaussian process. Along the way, we prove a quenched invariance principle for a random motion in the Howitt-Warren flow. Meanwhile, the centered quenched mean process of the random motion also converges on the scale $t^{1/4}$, where the limit is another universal Gaussian process.\\
 \\
 AMS 2000 subject classification. $60K35$, $60K37$, $60F17$
 \end{abstract}

 \begin{section}{Introduction}
 \begin{subsection}{Overview} 
 In the review article \cite{T.Sepp.2010}, Sepp\"al{\"a}inen discussed the processes of particle currents in several dynamical stochastic systems of particles on the one-dimensional integer lattice. It turns out that for independent random walks, independent random walks in an i.i.d space random environment, and the random average process (RAP), there is a universal limit for the current fluctuations on the scale $n^{1/4}$, which is a certain family of self-similar Gaussian processes. These three models all belong to the so-called Edwards-Wilkinson (EW) universality class. Two more recent examples in the EW class are one-dimensional Hammersley's harness process \cite{seppalainen2015hammersley} and the Atlas model \cite{dembo2015equilibrium}. In the EW class the limiting current fluctuations are described by the linear stochastic heat equation $Z_t=\upsilon Z_{xx}+\dot{W}$ where $\dot{W}$ is space-time white noise and $\upsilon$ is a non-zero parameter. In contrast, asymmetric simple exclusion process and a class of totally asymmetric zero range processes have nontrivial current fluctuations on the scale $n^{1/3}$, and the Tracy-Widom distributions are the universal limits. These two models belong to the Kardar-Parisi-Zhang (KPZ) universality class. More discussions about EW and KPZ universality classes and their relations can be found in \cite{T.Sepp.2010} and \cite{corwin2012kardar}. However, all the models that were shown to be in the EW universality up to now are discrete models defined on $\mathbb{Z}$. The motivation of this paper is to present a model in continuous space and time that also falls in the EW class.
 
 Recently, in \cite{le2004flows} Le Jan and Raimond introduced the so-called stochastic flow of kernels, which is a collection of random probability kernels. Heuristically, a stochastic flow of kernels can be interpreted as the transition kernels of a Markov process in a space-time random environment, where restrictions of the environment to disjoint space-time regions are independent and the law of the environment satisfies translation-invariance in space and time. Given the environment, one can sample $n$ independent Markov processes (random motions) and then average over the environment. This leads to a Markov process known as the $n$-point motion of the flow and their joint law satisfies a natural consistency condition: the marginal distribution of any $k$ components of an $n$-point motion is necessarily a $k$-point motion. The main result of Le Jan and Raimond \cite{le2004flows} is that any family of Feller processes that is consistent in this way gives rise to a unique stochastic flow of kernels. Using martingale problems, Howitt and Warren later constructed a class of consistent Feller processes on $\mathbb{R}$ which are Brownian motions with sticky interaction when they meet. Thus by the fundamental result of Le Jan and Raimond, this class of Feller process determines the unique stochastic flow of kernels which is now called the Howitt-Warren flow. In \cite{E.Sche.R.SunJ.Swart2014}, Schertzer, Sun and Swart showed that the Howitt-Warren flows can be realized as the transition kernels of a random motion in a space-time environment, constructed explicitly from the Brownian web and Brownian net. Thus the heuristic interpretation above naturally becomes rigorous.

 Dual smoothing process dual to Howitt-Warren flows, which is a function-valued process, was also introduced in \cite{E.Sche.R.SunJ.Swart2014}. As a continuum space-time analogue of RAP, dual smoothing process can be thought of as the evolution of the interface height function in a growth model as well. In one dimension, conservative interacting particle systems can always be equivalently formulated as interface models. Here the connection goes by regarding the gradient of the interface height function as a measure governing the distribution of the particles. The movement of particle currents can then be viewed as deposition or removal of particles from the growing interface. With such an equivalent formulation, the current process maps directly to the height function. We will show that on the scale $t^{1/4}$, the fluctuations of the height function (dual smoothing process), which is the first continuum space-time model shown to be in the EW class, converges weakly to a universal Gaussian process. Along the way, we will show that for random motions in the Howitt-Warren flows, the process of the centered quenched means, indexed by space and time, converges to a Gaussian process after rescaling by $t^{-1/4}$. Moreover, we will prove a quenched invariance principle for random motion in the Howitt-Warren flows, which is of independent interest. 
  \end{subsection}
 
  \begin{subsection}{Stochastic flows of kernels and Howitt-Warren flows}
  In this subsection, we recall the notion of a stochastic flow of kernels and the characterization of Howitt-Warren flows. We then state a quenched invariance principle for the random motion in a Howitt-Warren flow, as well as the first convergence theorem.
  
  We first give the definition of a stochastic flow of kernels as introduced in \cite{le2004flows}. Given a Polish space $E$, let $\mathcal{B}(E)$ be the Borel $\sigma$-field of $E$ and $\mathcal{M}_1(E)$ be the space of probability measures on $E$ equipped with the topology of weak convergence and the associated Borel $\sigma$-field. A random probability kernel on $E$ is a measurable function $K: \Omega\times E\times \mathcal{B}(E) \rightarrow\mathbb{R}$ such that $K^{\omega}(x,\cdot)\in\mathcal{M}_1(E)$, where $(\Omega,\mathcal{F},\mathbb{P})$ is the underlying probability space. We say that two random probability kernels $K$, $K'$ are equal in finite dimensional distributions if for any $n$ and $x_1,...,x_n\in E$, the distributions of the $n$-tuple of random probability measures $\big(K(x_1,\cdot),\cdots,K(x_n,\cdot)\big)$ and $\big(K'(x_1,\cdot),...,K'(x_n,\cdot)\big)$ are equal. We say that two or more random probability kernels are independent if their finite dimensional distributions are independent. Under these notations, Le Jan and Raimond (see \cite[Definition 1.6]{le2004flows}) defines:
  \begin{defn}{\textbf{$($Stochastic flow of kernels$)$}}\label{defn_SFK}
  A stochastic flow of kernels is a collection $(K_{s,t})_{s\leq t}$ of random probability kernels on the Polish space $E$ such that
  \begin{enumerate}
  \item[\rm(i)] For every $s\leq t\leq u$ and $x\in E$, almost surely, $K_{s,s}(x,dz)=\delta_x(dz)$ and \\ $\int K_{s,t}(x,dy)K_{t,u}(y,dz)=K_{s,u}(x,dz)$ .
  \item[\rm(ii)] For every $s\leq t$ and $u\in\mathbb{R}$, $K_{s,t}$ and $K_{s+u,t+u}$ are equal in finite dimensional distributions.
  \item[\rm(iii)] For any $t_0<\cdots <t_n$, the random probability kernels $(K_{t_{i-1},t_i})_{i=1}^n$ are independent.
  \end{enumerate}
  \end{defn}
  \begin{remark}\label{rmk1.2}{\rm
  In general, it is not known whether condition (i) can be strengthened to\\
  \begin{tabular}{ll}
  (i)' & A.s., $K_{s,s}(x,dz)=\delta_x(dz)$ and $\int K_{s,t}(x,dy)K_{t,u}(y,dz)=K_{s,u}(x,dz)$ for all $x\in E$ \\
  & and $s\leq t\leq u$,
  \end{tabular}\\
  so that $(K_{s,t})_{s\leq t}$ is a bona fide family of transition kernels of a random motion in a random space-time environment and the kernels satisfy the Chapman-Kolmogorov equation. However, for Howitt-Warren flows, this has been shown to be possible in \cite{E.Sche.R.SunJ.Swart2014}.}
  \end{remark}

  Given a stochastic flow of kernels $(K_{s,t})_{s\leq t}$, if we set $$P^{(n)}_{t-s}(\vec{x},d\vec{y}):= \mathbb{E}\big[ K_{s,t}(x_1,dy_1),\cdots,K_{s,t}(x_n,dy_n) \big]~~~~~(\vec{x},\vec{y}\in E^n, s\leq t),$$ then it defines a family of Markov transition probability kernels on $E^n$. We call the Markov process with these transition probabilities the $n$-point motion associated with the stochastic flow of kernels $(K_{s,t})_{s\leq t}$ and a natural consistent condition is satisfied. Conversely, a fundamental theorem of Le Jan and Raimond \cite[Theorem 2.1]{le2004flows} shows that any consistent family of Feller processes on a locally compact space $E$ gives rise to a stochastic flow of kernels on $E$ and it is unique in the sense that any two versions of such stochastic flows of kernels are equal in finite dimensional distributions.

  Howitt and Warren constructed in \cite{howitt2009consistent} a consistent family of Feller processes on $\mathbb{R}$ via a well posed martingale problems, which are Brownian motions with drift $\beta\in\mathbb{R}$ and sticky interactions that can be characterized by a finite measure $\mu$ on $[0,1]$. The associated stochastic flow of kernels is now called the Howitt-Warren flow. The associated $n$-point motion evolves as $n$ independent Brownian motions with the same drift when they do not coincide, but it is possible that two or more Brownian motions may meet at the same location because of the stickiness which makes the $n$-point motion spend positive Lebesgue time together. In \cite[Proposition 2.3]{E.Sche.R.SunJ.Swart2014}, it was shown that for the Howitt-Warren flows, one can choose a set of probability one on which the relations in Definition 1.1 (i) holds for all $s\leq t\leq u$ and $x\in E$, as pointed out in Remark \ref{rmk1.2}.
  
  Since the formal formulation of Howitt-Warren $2$-point motion satisfies the purpose of this paper, we only recall the definition of the $2$-point motion. For the definition of Howitt-Warren martingale problems, we refer to either \cite[Definition 2.1]{howitt2009consistent} or \cite[Definition 2.2]{E.Sche.R.SunJ.Swart2014}.
  \begin{defn}{\textbf{$($Howitt-Warren $2$-point motion$)$}}
  \label{2pointmotion}
  A Howitt-Warren $2$-point motion is an $\mathbb{R}^2$-valued process $\vec{X}=(X^1(t),X^2(t))_{t\geq 0}$ where $(X^1(t))_{t\geq 0}$ and $(X^2(t))_{t\geq 0}$ are two Brownian motions with some drift $\beta\in \mathbb{R}$, the covariation process of $X^1$ and $X^2$ is given by
  \begin{equation}
  \langle X^1,X^2\rangle (t)=\int_{0}^{t} 1_{\{X^1(s)=X^2(s)\}}ds,
  \end{equation}
  and there exists $\nu\in(0,\infty)$, called the stickiness parameter, such that
  \begin{equation}
  \nu |X^1(t)-X^2(t)|- \int_{0}^{t} 1_{\{X^1(s)=X^2(s)\}}ds
  \end{equation}
  is a martingale with respect to the filtration generated by $\vec{X}$.
  \end{defn}

  \begin{remark}{\rm
  Here we need only one parameter $\nu$ to characterize the stickiness of $X^1$ and $X^2$ when they intersect instead of a finite measure $\mu$ needed to characterize the $n$-point motions as in \cite{E.Sche.R.SunJ.Swart2014}, where $\nu=1/\big(4\mu ([0,1])\big)$. Moreover, $X^1(t)-X^2(t)$ is the well-known sticky Brownian motion which can be obtained by time-changing a standard Brownian motion in such a way that at the origin its local time becomes $1/2\nu$ times the real time, and it behaves as a standard Brownian motion on $\mathbb{R}\backslash\{0\}$. In \cite{howitt2009consistent}, such $\vec{X}$ is called $\theta$-coupled Brownian motions with $\theta=1/2\nu$. For the Howitt-Warren $2$-point motion, we will give a SDE representation in the next section.}
  \end{remark}
  
  Given a realization of the Howitt-Warren flow $(K^{\omega}_{s,t})_{s\leq t}$, one can sample a set of independent random motions $(X^1_t,\cdots,X^n_t)$. We let $\mathbb{P}$ (resp.~$\mathbb{E}$) denote the probability (resp.~expectation) for the environment $\omega$, let $P^{\omega}$ (resp.~$E^{\omega}$) denote the quenched law (resp.~quenched expectation) for the random motions given the environment $\omega$, and let $P:=\mathbb{E}P^{\omega}(\cdot)$ (resp.~$E$) denote the averaged law (or annealed law) (resp. averaged expectation) for the random motions by integrating out the environment. Under this notation, for example, two random motions $(X^1_t,X^2_t)$ independent under the law $P^{\omega}$ are in fact a Howitt-Warren $2$-point motion under the averaged law $P$.
  
  If we consider a random motion $(X_t)_{t\geq 0}$ starting from the origin in the Howitt-Warren flow $(K^{\omega}_{s,t})_{s\leq t}$ with drift $\beta$ and characteristic measure $\mu$ (so that the stickiness parameter for the $2$-point motion is $\nu=1/4\mu ([0,1])$), then our first result is an almost sure quenched invariance principle for $(X_t)_{t\geq 0}$, which is analogous to the one for the random walk in i.i.d space-time random environment (\cite[Theorem 1]{F.Rass.T.Sepp.2005}).
  \begin{thm} \label{qIP}
   Let $Y_t:=X_t-\beta t$, then for $\mathbb{P}$-a.e.$\!$ $\omega$, the process $(Y_{nt}/\sqrt{n}) _{t\geq 0}$ converges weakly to a standard Brownian motion as $n\rightarrow\infty$. Moreover, for $\mathbb{P}$-a.e.$\!$ $\omega$, $n^{-1/2}\max_{s\leq nt}\big|E^{\omega}X_s-\beta s\big|$ converges to $0$, and therefore the same quenched invariance principle also holds for the process $\tilde{Y}_t:=X_t-E^{\omega}[X_t]$.
  \end{thm}
  \noindent Since $\mathbb{P}$ is invariant w.r.t. the space-time shift of the environment $\omega$, this invariant principle holds for the random motion starting from any space-time point.
  
  If we use the superscript to represent the starting point of the random motion, i.e., $(X^{x_0,t_0}_t)_{t\geq t_0}$ is a random motion starting from the space-time point $(x_0,t_0)$, then we can state our second result:
     
  \begin{thm} \label{GaussianlimitThm}
  For every $(t,r)\in\mathbb{R}^+\times\mathbb{R}$, define two rescaled centered quenched means as follows:
  \begin{eqnarray}
  & a_n(t,r):=n^{-1/4}\left(E^{\omega}\big[X^{r\sqrt{n}-\beta nt,-nt}_0\big]-r\sqrt{n}\right),\\
  & b_n(t,r)=n^{-1/4}\left(E^{\omega}\big[X^{r\sqrt{n},0}_{nt}\big]-r\sqrt{n}-\beta nt\right),
  \end{eqnarray}
  then the finite dimensional distributions of the processes $\{a_n(t,r):(t,r)\in\mathbb{R}^+\times\mathbb{R}\}$ and $\{b_n(t,r):(t,r)\in\mathbb{R}^+\times\mathbb{R}\}$ converge to those of the Gaussian processes $\{a(t,r):(t,r)\in\mathbb{R}^+\times\mathbb{R}\}$ and $\{b(t,r):(t,r)\in\mathbb{R}^+\times\mathbb{R}\}$ with covariance functions given by $\Gamma((t,r),(s,q))\!$ and $\Gamma((t\wedge s,r),(t\wedge s,q))$ respectively, where
  \begin{equation}\label{CovFun}
  \Gamma((t,r),(s,q)):=\nu \int_{|t-s|}^{t+s}\frac{1}{\sqrt{\pi u}}e^{-\frac{(r-q)^2}{2u}}du.
  \end{equation}
  \end{thm}
  \begin{remark}\label{rmk1.7}{\rm
  (i) We will only give the proof of the convergence of $\big(a_n(t,r)\big)$, since all the ingredients and arguments needed in the proof for $\big(b_n(t,r)\big)$ are essentially the same. Indeed, if we define the translation $T_{t,x}$ of the random environment that makes $(x,t)$ the new space-time origin, then it is easy to see that $b_n(t,r)=a_n(t,r)\circ T_{nt,\beta nt}$. 
  (ii) For any $(t,r)\in\mathbb{R}^+\times\mathbb{R}$, $a_n(t,r)$ is a random variable of the environment between time $-nt$ and $0$. 
  (iii) The variance of the quenched mean process is of order $n^{1/2}$ (see Lemma \ref{variance} (ii)), and this leads to the choice of the scale $n^{-1/4}$ in $a_n(t,r)$ and $b_n(t,r)$.}
  \end{remark}
  
  \end{subsection}
  \begin{subsection}{Dual smoothing process}
  Given a Howitt-Warren flow $(K^{\omega}_{s,t})_{s\leq t}$, a Howitt-Warren process, which is a measure-valued Markov process, is defined by
  \begin{equation}
    \rho_t(dy):=\int \rho_0(dx)K^{\omega}_{0,t}(x,dy)~~~~~(t\geq 0),
  \end{equation} where $\rho_0$ a finite measure on $\mathbb{R}$. A function-valued dual smoothing process is defined by
  \begin{equation}\label{DSP}
    \zeta_t(x):=\int K^{\omega}_{-t,0}(x,dy)\zeta_0(y)~~~~~~(x\in\mathbb{R},~~ t\geq 0),
  \end{equation} where $\zeta_0\in D_b(\mathbb{R})$, the space of bounded c\'adl\'ag functions on $\mathbb{R}$. These two processes are shown to be dual to each other in \cite[Lemma 11.1]{E.Sche.R.SunJ.Swart2014}. Indeed, from $(K^{\omega}_{s,t})_{s\leq t}$ one can define a dual Howitt-Warren process $(\hat{\rho}_t)_{t\geq 0}$, and regard $\zeta_t$ as its height function at time $t$. To see this fact at a heuristic level, we begin with the description of the discrete Howitt-Warren flow.
  
  Let $\mathbb{Z}^2_{even}\!\!:=\{(x,t)\!:x,t\in\mathbb{Z},x+t~{\rm is~even}\}$, where the first and second coordinates are interpreted as space and time. Let $\omega\!:=\!(\omega_z)_{z\in\mathbb{Z}^2_{even}}$ be i.i.d.$\!$ $[0,1]$-valued random variables with common distribution $\mathbb{Q}$. We view $\omega$ as a random space-time environment for a random walk. That is, conditional on $\omega$, if a random walk is at time $t$ at the position $x$, then in the next unit time step the walk jumps to $x+1$ with probability $\omega_{(x,t)}$ and to $x-1$ with the remaining probability $1-\omega_{(x,t)}$.
  
  If we use $Q^{\omega}_{(s,x)}$ to denote the quenched law of a random walk $X:=(X_t)_{t\geq s}$ starting from the space-time point $(x,s)$, then setting $\bar{K}^{\omega}_{s,t}(x,y):=Q^{\omega}_{(s,x)}\big(X(t)\!=\!y\big)$ defines the discrete Howitt-Warren flow $(\bar{K}^{\omega}_{s,t})_{s\leq t}$. It is shown in \cite{E.Sche.R.SunJ.Swart2014} that with suitable assumption the discrete flow under diffusive scaling converges to the Howitt-Warren flow, and a graphical construction analogous to the discrete one can be carried over to the continuum level. A natural corollary of this graphical construction is that in Definition 1.1, condition (i) can be strengthened to (i)' for Howitt-Warren flows.
  
  Given a realization of the environment, one can sample coalescing random walks starting from every point in $\mathbb{Z}^2_{even}$, which is called a (random) discrete web. Moreover, one can couple a dual discrete web in the following way. Consider the coalescing random walks running backwards in time, starting from every point in $\mathbb{Z}^2_{odd}:=\mathbb{Z}^2\setminus\mathbb{Z}^2_{even}$. For each $(x,t+1)\in\mathbb{Z}^2_{odd}$, the backward walk at time $t+1$ at the position $x$ jumps to $(x-1,t)$ if the forward walk in the discrete web jumps from $(x,t)$ to $(x+1,t+1)$, and otherwise the backward walk jumps to $(x+1,t)$. It is then easy to see that the law of the dual discrete web determines the dual discrete Howitt-Warren flow, and hence the dual Howitt-Warren flow $(\hat{K}_{t,s})_{t\geq s}$, which is equal in distribution to the Howitt-Warren flow. Furthermore, noting the non-crossing property (i.e., in the coupling, random walks in the discrete web does not cross any random walk in the dual web), we have the following relationship:
  \begin{equation} \label{KhatK}
  K_{s,t}\big(x,[y,\infty) \big)=\hat{K}_{t,s}\big(y,(-\infty,x]\big)~~~~~(x,y\in\mathbb{R}, s\leq t).
  \end{equation}
  In other words, if we sample a forward random motion $X:=(X_u)_{u\geq s}$ from the space-time point $(x,s)$ in $(K_{u,v})_{u\leq v}$ and a backward random motion $\hat{X}:=(\hat{X}_u)_{u\leq t}$ from $(y,t)$ in $(\hat{K}_{v,u})_{v\geq u}$, and if $s\leq t$ and $X_t\geq y$, then since the paths of $X$ and $\hat{X}$ do not cross each other, we must have $\hat{X}_s\leq x$. Note that for deterministic $y$, the probability that $X_t=y$ is zero, therefore in (\ref{KhatK}) we can change the closed interval to open interval.
  
  Now if we consider the dual Howitt-Warren process with finite initial measure $\hat{\rho}_0$,
  \begin{equation}
  \hat{\rho}_t(dy):=\int \hat{\rho}_0(dx)\hat{K}_{0,-t}(x,dy),
  \end{equation}
  then setting $\zeta_0$ in (\ref{DSP}) as the height function of $\hat{\rho}_0$ by
  \begin{equation}
  \zeta_0(x)=\int_{(-\infty,x]}\hat{\rho}_0(dy)~~~~~(x\in\mathbb{R}),
  \end{equation}
  we have that $\zeta_t$ is the height function of $\hat{\rho}_t$, and by (\ref{KhatK}) the current flow of $\hat{\rho}$ over the line segment from $(0,x)$ to $(-t,y)$ is:
  \begin{eqnarray}
  \nonumber && \int_{(x,\infty)}\hat{\rho}_0(dz)\hat{K}_{0,-t}\big(z,(-\infty,y]\big) - \int_{(-\infty,x]}\hat{\rho}_0(dz)\hat{K}_{0,-t}\big(z,(y,\infty)\big) \\
  \nonumber&=& \int_{(x,\infty)}\Big\{\int_{z}^{\infty}K_{-t,0}(y,dw)\Big\}\hat{\rho}_0(dz) - \int_{(-\infty,x]}\Big\{\int_{-\infty}^{z}K_{-t,0}(y,dw)\Big\}\hat{\rho}_0(dz) \\
  \nonumber&=& \int\Big\{\int_{x}^{w}\hat{\rho}_0(dz)\Big\}K_{-t,0}(y,dw) \\
  &=& \zeta_t(y)-\zeta_0(x).
  \end{eqnarray}
  As a result, considering the current fluctuations of the dual Howitt-Warren process is equivalent to considering the fluctuations of the dual smoothing process.  
    
  Now we consider the fluctuations of a class of generalized dual smoothing processes. For any deterministic point $x_0\in\mathbb{R}$, we look at the fluctuation rescaled by $n^{-1/4}$ along the characteristic line $x(t)=x_0-\beta t$, namely the quantity
   \begin{equation}\label{DSPzn}
   z_n(t,r):=n^{-1/4}\left\{\zeta^{(n)}_{nt}(nx_0+r\sqrt{n}-\beta nt)-\zeta^{(n)}_0(nx_0+r\sqrt{n})\right\}.
   \end{equation}
   
   For our purpose, we assume the following initial condition:
   \begin{assumption}
   For each $n\in\mathbb{N}$, define an initial condition $\zeta^{(n)}_0(x)=f^{(n)}(x)+W(x)$, where $\big(W(x)\big)_{x\in\mathbb{R}}$ is a two-sided Brownian motion, independent of the Howitt-Warren flow, with $W(0)=0$, and $f^{(n)}(x):=nf(\frac{x}{n})$, where $f$ is a $C^1$ function such that $f(0)=0$, $f'$ is bounded and satisfies the following H\"older continuity condition: there exist constants $C>0$ and $\gamma>1/2$ such that
   \begin{equation}\label{6A}
   \left|f'(x)-f'(y)\right|<C\left|x-y\right|^\gamma~~~~~~~~(x,y\in\mathbb{R}).
   \end{equation}
   \end{assumption}
   
   It turns out that under Assumption I, as $n$ tends to $\infty$, $\{z_n(t,r):(t,r)\in\mathbb{R}^+\times\mathbb{R}\}$ converges to a Gaussian process $\{z(t,r):(t,r)\in\mathbb{R}^+\times\mathbb{R}\}$ in finite dimensional distributions. So we next describe the limiting process.
   
   Define a covariance function $\Gamma_0$ on $(\mathbb{R}^+ \times \mathbb{R} ) \times(\mathbb{R}^+\times \mathbb{R})$,
   \begin{eqnarray}\label{1.14}
   \nonumber\Gamma_0((t,r),(s,q)) &:=& \int_{r\vee q}^{\infty} P\left(B(t) > z-r \right)P\left(B(s) > z-q \right)dz \\
    \nonumber&&-1_{\{r>q\}} \int_{q}^{r} P\left(B(t) < z-r \right) P\left(B(s) > z-q \right)dz \\
    \nonumber&&-1_{\{r<q\}} \int^{q}_{r} P\left(B(t)> z-r \right)P\left(B(s) < z-q \right)dz \\
    && +\int^{r\wedge q}_{-\infty} P\left(B(t) < z-r \right)P\left(B(s) < z-q \right)dz
   \end{eqnarray}
   where $B$ is a standard Brownian motion,
   and recall the covariance function $\Gamma$ defined in (\ref{CovFun}):
   \begin{eqnarray}
   \Gamma((t,r),(s,q))=\nu \int_{|t-s|}^{t+s}\frac{1}{\sqrt{\pi u}}e^{-\frac{(r-q)^2}{2u}}du.
   \end{eqnarray}
   Then $\{z(t,r):(t,r)\in\mathbb{R}+\times\mathbb{R}\}$ is a mean zero Gaussian process with covariance given by 
   \begin{equation}\label{DSPlimitcov}
   Ez(t,r)z(s,q) = f'^2(x_0)\Gamma((t,r),(s,q))+ \Gamma_0\left((t,r),(s,q)\right).
   \end{equation}
   In fact, the first term in the right-hand side of (\ref{DSPlimitcov}) comes from the fluctuation caused by the dynamics, while the second is from the initial noise $\zeta_0$.
   
   \begin{thm}\label{final}
   Under Assumption I, the finite dimensional distributions of the current fluctuations $\{z_n(t,r):(t,r)\in\mathbb{R}^+\times\mathbb{R}\}$ defined in (\ref{DSPzn}) converge weakly to those of the mean zero Gaussian process $\{z(t,r):(t,r)\in\mathbb{R}^+\times\mathbb{R}\}$ with covariance function \eqref{DSPlimitcov}.
   \end{thm}
  
  In contrast to the discrete random average process (RAP), we have used some different strategies to overcome the technical difficulties in the continuum case. The quenched invariance principle of the random walk in an i.i.d.$\!$ space-time random environment, which is a discrete analogue of Theorem \ref{qIP}, was previously proved based on the view of the particle and martingale techniques in \cite{F.Rass.T.Sepp.2005}. However, this method can not be transferred to the continuum case easily. Instead, we applied the second moment method to show Theorem \ref{qIP}, which is also efficient for the discrete case (we provide the proof in Appendix A). Furthermore, we take advantage of self-duality of the Howitt-Warren flows and different carefully-chosen couplings to approach several estimates, which turn out to be more difficult than the discrete case. Besides, these techniques could be useful for stochastic flows of kernels and Brownian web (see \cite{FontesIsopiNewmanRavishankar2004}).

  The rest of the paper is organized as follows. Section 2 provides the SDEs for the Howitt-Warren $2$-point motion, establishes some properties of their collision local time, and computes the covariance of the 2-point motion, which are served as preliminaries of the main proofs. Section 3 proves the quenched invariance principle for the random motion in the Howitt-Warren flows. Section 4 and Section 5 prove Theorem \ref{GaussianlimitThm} and Theorem \ref{final} respectively. In Appendix A, we provide a proof of the quenched invariance principle of the random motion in an i.i.d.$\!$ space-time random environment via the second moment method (a result of independent interest).
  \end{subsection}

 \end{section}
 \begin{section}{Howitt-Warren $2$-point motion preliminaries}

 In this section we first give a set of SDEs that characterizes the Howitt-Warren 2-point motion, which are Brownian motions with sticky interactions when they meet. We then derive two useful lemmas about the sticky Brownian motion. Lastly we consider the covariance of the 2-point motion.

  \begin{subsection}{SDEs for the Howitt-Warren $2$-point motion}
  We first recall the concept of local time of continuous local martingale, and then give a set of SDEs which has a unique weak solution. We will show that a Howitt-Warren $2$-point motion can be represented by the weak solution of the SDEs. This representation plays an important role throughout the paper.

  In Section 3.7 of \cite{I.Kara.S.E.Shre.1991}, local time of continuous semimartingale is discussed. It is a generalized concept of local time of Brownian motion, first introduced by P. L\'evy, which is used to measure the time that a Brownian motion spends in the vicinity of a deterministic point. Here we only need to consider the local time $\Lambda_{x_0}(t,x)$ of local martingale, and we list some properties of $\Lambda_{x_0}(t,x)$ as a proposition. For further theory of local time, we refer to \cite[Chapter 3]{I.Kara.S.E.Shre.1991}.
  \begin{prop}\label{localtime}
  Let $X_t=x_0+M_t$ be a continuous local martingale on some probability space $(\Omega,\mathcal{F},\mathbb{P})$, where $X_0=x_0\in\mathbb{R}$, and $(M_t)_{t\geq 0}$ with $M_0=0$ is adapted to a filtration $(\mathcal{F}_t)_{t\geq 0}$. Then there exists an a.s. unique process $\Lambda_{x_0}(t,x)$, which is called the local martingale local time, defined on $\mathbb{R}^+\times\mathbb{R}\times \Omega$, such that the following holds:
  \begin{enumerate}
  \item[\rm(i)] For all $(t,x,\omega)\in \mathbb{R}^+\times\mathbb{R}\times \Omega$, $\Lambda_{x_0}(t,x)(\omega)$ is nonnegative.
  \item[\rm(ii)] For every fixed $x\in\mathbb{R}$, $\Lambda_{x_0}(0,x)=0$,  $\Lambda_{x_0}(t,x)$ is continuous and nondecreasing in $t$, and
   \begin{equation}
   \int_{0}^{\infty} 1_{\mathbb{R}\backslash\{x\}} (X_t)\Lambda_{x_0}(dt,x)=0,~~~ {\rm for} ~\mathbb{P}\!-\!a.e.~\omega\in\Omega.
   \end{equation}
  \item[\rm(iii)] $($Tanaka-Meyer formula$)$ For every fixed $x\in\mathbb{R}$,
  \begin{equation}
  |X_t-x|=|X_0-x|+\int_{0}^{t} {\rm sgn}(X_s-x)dM_s + 2\Lambda_{x_0}(t,x),
  \end{equation}
  where ${\rm sgn}(x)$ is the sign function.
  \item[\rm(iv)] If $X_t$ is a Brownian motion $B_t$ with $B_0=x_0$ and $\mathbb{E}[B_t^2]=\sigma^2t$ (usually we use the notation $L_{x_0}(t,x)$ instead of $\Lambda_{x_0}(t,x)$ in this case), then for every measurable function $f:\mathbb{R}\rightarrow [0,\infty)$, we have that a.s.,
  $$\sigma^2\int_{0}^{t}f(B_s)ds=2\int_{-\infty}^{\infty}f(x)L_{x_0}(t,x)dx,$$
  and $L_{x_0}(t,x)$ is continuous in $(t,x)$.
  \end{enumerate}
  \end{prop}
  
  Later, when the starting point of $X_t$ is clear, we will abbreviate the notation $\Lambda_{x_0}(t,x)$ by $\Lambda(t,x)$.

  Now we consider the following SDEs:
  \begin{equation} \label{spde}
  \begin{array}{rcl}
    dX^1_t & = & 1_{ \{X^1_t\neq X^2_t\}}dB^1_t+1_{\{X^1_t=X^2_t\}}dB^3_t+\beta dt, \\
    dX^2_t & = & 1_{ \{X^1_t\neq X^2_t\}}dB^2_t+1_{\{X^1_t=X^2_t\}}dB^3_t+\beta dt, \\
    1_{\{X^1_t= X^2_t\}}dt & = & 2\nu \Lambda(dt,0),
  \end{array}
  \end{equation}
  with initial condition $X^1_0=x_1$ and $X^2_0=x_2$.
  Here $\{B^i_t;~i=1,2,3\}$ are independent standard Brownian motions starting form the origin, $\nu$ is a constant parameter (later we will see that $\nu$ coincides with the stickiness parameter given in Definition \ref{2pointmotion}), and  $\Lambda(t,x)$ is the local time of the difference process $X^1_t-X^2_t$. Note that from the first two equations, $X^1_t-X^2_t$ must be a local martingale, which leads to the existence of $\Lambda(t,x)$ by Proposition \ref{localtime}. In particular, $\Lambda(t,0)$ is continuous and nondecreasing in $t$ for a.e. $\!\omega$. Consequently, it induces a  measure on $\mathbb{R}^+$ and the third equation of (\ref{spde}) is meaningful.
  
  The SDEs (\ref{spde}) gives a representation of the Howitt-Warren 2-point motion as stated in the following theorem.
  \begin{thm} \label{thmsde}
  The SDEs \eqref{spde} is well posed, i.e., for every initial condition $(x_1,x_2)\in\mathbb{R}^2$, \eqref{spde} admits a weak solution which is unique in law. Furthermore, any Howitt-Warren $2$-point motion $(X_t^1,X_t^2)$ is a solution of \eqref{spde}, and vice versa.
  \end{thm}  
  We prove this theorem by the following lemmas.

  \begin{lem}\label{sdeexistence}
  Given initial condition $X^1_0=x_1$, $X^2_0=x_2$ for any $x_1,x_2\in\mathbb{R}$, the SDEs \eqref{spde} has a weak solution, that is, there is a quintuple $(X^1,X^2,B^1,B^2,B^3)$ and a filtration $\{\mathcal{F}_t\}_{t\geq 0}$ such that the quintuple is adapted to $\{\mathcal{F}_t\}_{t\geq 0}$, $B^1,B^2,B^3$ are independent Brownian motions and $(X^1,X^2)$ satisfies \eqref{spde} in It$\hat{o}$-integral form.
  \end{lem}
  \begin{proof}
  Let $\{\tilde{B}_t^i,\hat{B}_t^i;~i=1,2,3\}$ be independent standard Brownian motions starting from $0$ and $\{\mathcal{F}_t\}_{t\geq 0}$ be the filtration generated by these Brownian motions. Define $W_t:=x_1+ \hat{B}_t^1-x_2-\hat{B}_t^2$, then $W_t$ is a Brownian motion and let $L(t,x)$ denote the local time of $W_t$. 
  
  Set $A_t:=t+2\nu L(t,0)$, then $A_t$ is a strictly increasing and continuous function and $A_t\geq t$. Therefore, we can define the inverse function of $A_t$ by $T_t:=A_t^{-1}$ and define also $S_t:=t-T_t$. Now let
  \begin{eqnarray}
  X_t^i:=x_i+\hat{B}_{T_t}^i+ \hat{B}_{S_t}^3+\beta t,~~~~~~i=1,2.
  \end{eqnarray}
  Define then
  \begin{eqnarray}
  \begin{array}{rcl}
  B_t^1&:=&\hat{B}_{T_t}^1+\int_{0}^{t}1_{\{X_s^1= X_s^2\}} d\tilde{B}_s^1, \\
  B_t^2&:=&\hat{B}_{T_t}^2+\int_{0}^{t}1_{\{X_s^1= X_s^2\}} d\tilde{B}_s^2, \\
  B_t^3&:=&\hat{B}_{S_t}^3+\int_{0}^{t}1_{\{X_s^1 \neq X_s^2\}} d\tilde{B}_s^3.
  \end{array}
  \end{eqnarray}
  We claim that the quintuple $(X_t^1,X_t^2,B_t^1,B_t^2,B_t^3)$ together with the filtration $(\mathcal{F}_t)_{t\geq 0}$ is a weak solution of (\ref{spde}).

  To see this, first we note that the quintuple is adapted to $\{\mathcal{F}_t\}_{t\geq 0}$. Next we are going to prove that $B^1,B^2,B^3$ are independent Brownian motions by L\'evy's characterization of Brownian motion. Since
  \begin{equation*}
  \mathbb{E}\left[(\int_{0}^{T_t}1_{\{W_{s}=0\}}d\hat{B}_s^1 )^2\right] = \mathbb{E}\left[\int_{0}^{T_t}1_{\{W_{s}=0\}}ds\right] =0,
  \end{equation*}
  so $\int_{0}^{T_t}1_{\{W_{s}=0\}}d\hat{B}_s^1=0$ a.s..
  Combining this with the fact that $W_{T_t}=0$ if and only if $X_t^1=X_t^2$, we have a.s.,
  \begin{equation}
  \hat{B}_{T_t}^1=\int_{0}^{T_t}1_{\{W_{s}\neq 0\}}d\hat{B}_s^1= \int_{0}^{t}1_{\{W_{T_s\neq 0}\}}d\hat{B}_{T_s}^1 = \int_{0}^{t}1_{\{ X_s^1\neq X_s^2\}}d\hat{B}_{T_s}^1,
  \end{equation}
  and the quadratic variation
  \begin{eqnarray}\nonumber
  \langle \hat{B}_{T}^1 \rangle_t & =& T_t = \int_{0}^{T_t} 1_{\{W_{s}\neq 0\}}ds =      \int_{0}^{T_t} 1_{\{W_{s}\neq 0\}} \Big(ds+2\nu L(ds,0)\Big) \\
   & = &\int_{0}^{T_t} 1_{\{W_{s}\neq 0\}}dA_s =\int_{0}^{t} 1_{\{W_{T_s}\neq 0\}}ds = \int_{0}^{t}1_{\{ X_s^1\neq X_s^2\}}ds,
  \end{eqnarray}
  where the third equality holds because of Proposition \ref{localtime} (ii). Hence by the independence of $\hat{B}^1$ and $\tilde{B}^1$, the quadratic variation of $B^1$ is given by
  \begin{equation}
  \langle B^1\rangle_t=T_t+\int_{0}^{t}1_{\{ X_s^1=X_s^2\}}ds=\int_{0}^{t}1ds=t.
  \end{equation}
  Notice that $B_t^1$ is a continuous martingale with respect to $(\mathcal{F}_t)_{t\geq 0}$, so by L\'evy's characterization $B_t^1$ is a Brownian motion. Similarly, $B_t^2$ is also a Brownian motion. As to $B_t^3$, we have
  \begin{equation}
  \langle B^3\rangle_t=S_t+\int_{0}^{t}1_{\{ X_s^1\neq X_s^2\}}ds=S_t+T_t=t,
  \end{equation}
  which implies that $B_t^3$ is also a Brownian motion. It is not difficult to see the independence of $B_t^1$, $B_t^2$ and $B_t^3$ since the covariation process of each two is zero.

  Moreover, by the construction of $X^1_t$, $X^2_t$, $B_t^1$, $B_t^2$ and $B_t^3$, for $i=1,2$,
  \begin{eqnarray}
  \nonumber X_t^i=\! x_i\!+\!\!\int_{0}^{t}1_{\{ X_s^1\neq X_s^2\}}d\hat{B}_{T_s}^i \!+\!\! \int_{0}^{t}1_{\{ X_s^1= X_s^2\}}d\hat{B}_{S_s}^3 \!+\!\beta t = \!x_i\! +\!\! \int_{0}^{t}1_{\{ X_s^1\neq X_s^2\}}dB_s^i \!+\!\! \int_{0}^{t}1_{\{ X_s^1= X_s^2\}}dB_s^3 \!+\!\beta t.
  \end{eqnarray}
  Thus $X_t^1$ and $X_t^2$ solve the first two equations of (\ref{spde}).

  It remains to show that  $(X_t^1,X_t^2)$ satisfies the third equation in (\ref{spde}). Since $X_t^1-X_t^2=W_{T_t}$, we have $\Lambda(t,x)=L(T_t,x)$. Observe that
  \begin{equation}
  \begin{split}
  \int_{0}^{t}1_{\{ X_s^1= X_s^2\}}ds & =\int_{0}^{t}1_{\{W_{T_s}=0\}}ds =\int_{0}^{T_t}1_{\{W_s=0\}}dA_s 
    = \int_{0}^{T_t}1_{\{W_s=0\}} \Big(ds+2\nu L(ds,0)\Big)  \\
   & = 2\nu \int_{0}^{T_t}1_{\{W_s=0\}}L(ds,0)  
    = 2\nu L(T_t,0) = 2\nu \Lambda(t,0).
  \end{split}
  \end{equation}

  Thus, $(X_t^1,X_t^2,B_t^1,B_t^2,B_t^3)$ solves the SDEs (\ref{spde}).
  \end{proof}
  
  Moreover, the weak solution of (\ref{spde}) is the Howitt-Warren 2-point motion.
  \begin{lem}
  Given initial condition $X^1_0=x_1$, $X^2_0=x_2$ for any $x_1,x_2\in \mathbb{R}$, the solution $(X^1_t,X^2_t)$ $($in weak sense$)$ solves the martingale problem for the Howitt-Warren $2$-point motion as defined in Definition \ref{2pointmotion}.
  \end{lem}
  \begin{proof}
  Suppose that $(X^1_t,X^2_t)$ is a solution of (\ref{spde}). Then
  $$X_t^1-x_1-\beta t= \int_{0}^{t} 1_{ \{X^1_s\neq X^2_s\}}dB^1_s+\int_{0}^{t} 1_{\{X^1_s=X^2_s\}}dB^3_s,$$
  where $B^1_t$ and $B^3_t$ are independent Brownian motions. It is easy to see that $X_t^1-\beta t$ is a continuous martingale and the quadratic variation is $t$. By L\'evy's characterization $X_t^1-\beta t$ is a Brownian motion, and so is $X_t^2-\beta t$.

  Furthermore, applying the Tanaka-Meyer formula to the martingale $X_t^1-X_t^2$, we have
  \begin{equation}
  |X_t^1-X_t^2|=|x_1-x_2|+\int_{0}^{t}{\rm sgn} (X_s^1-X_s^2)d(X_t^1-X_t^2)+2\Lambda(t,0).
  \end{equation}
  Since $\int_{0}^{t}{\rm sgn}(X_s^1-X_s^2)d(X_t^1-X_t^2)$ is a martingale and
  $2\Lambda(t,0)=\frac{1}{\nu}\int_{0}^{t} 1_{ \{X^1_s= X^2_s\}}ds$ by (\ref{spde}), we conclude that $\nu|X_t^1-X_t^2|-\int_{0}^{t} 1_{ \{X^1_s= X^2_s\}}ds$ is a martingale.

  Therefore, the solution $(X^1_t,X^2_t)$ solves the martingale problem for the $2$-point motion.
  \end{proof}
  Lastly, the uniqueness of the Howitt-Warren $2$-point motion follows from the uniqueness of the Howitt-Warren martingale problems shown in \cite{howitt2009consistent}. Therefore, the second statement of Theorem \ref{thmsde} holds, and the solution of (\ref{spde}) is unique. Thus, we have proved Theorem \ref{thmsde}.
  
  From now on, we will identify the Howitt-Warren $2$-point motion and the solution of (\ref{spde}), since we are only interested in the distribution of the 2-point motion.
  \end{subsection}
  
  \begin{subsection}{Local time preliminaries}
  In this subsection, we always let $(X_t^1,X_t^2)$ be a Howitt-Warren $2$-point motion starting from $(x_1,x_2)$, and $\nu$ be the stickiness parameter. We will derive the first moment of the local time $\Lambda_{x_1-x_2}(t,0)$ of the difference process $X_t^1-X_t^2$ at the origin ($X_t^1-X_t^2$ is indeed a sticky Brownian motion). This result will be used for several times in Section 2.3, 3 and 4. We will also estimate the probability of $X_t^1-X_t^2$ visiting the origin between two fixed times, which will be applied in the proof of Theorem \ref{GaussianlimitThm}.
  
  \begin{lem}\label{lem-guji1}
  For all $x_1,x_2\in\mathbb{R}$, we have $\mathbb{E}[\Lambda_{x_1-x_2}(t,0)]=O(t^{1/2})$, where $f(t)=O\big(g(t)\big)$ denotes that there exists a constant $C>0$ such that $f(t)\leq Cg(t)$ as $t\rightarrow\infty$. Moreover, if $x_1-x_2=0$, then
  \begin{equation}\label{2.13}
  \mathbb{E}[\Lambda_0(t,0)]=\sqrt{\frac{2}{\pi}} t^{1/2} + \left(2 e^{t/2\nu^2}\left[1-\Phi\Big(\frac{\sqrt{t}}{\nu}\Big)\right] - 1\right)\nu ,
  \end{equation}
  where $\Phi(x):=\int_{-\infty}^{x}\frac{1}{\sqrt{2\pi}}e^{-y^2/2}dy$.
  \end{lem}
  \begin{proof}
  We follow the notations in Lemma \ref{sdeexistence}, and first derive the distribution of the local time $\Lambda(t,0)$. In the proof of Lemma \ref{sdeexistence}, we have the relation $X_t^1-X_t^2 =W_{T_t}$, where $W_t$ is a Brownian motion with $\mathbb{E}[W^2_t]=2t$ starting from $x_1-x_2$ and $T_t$ is the time change defined in Lemma \ref{sdeexistence}. For the local time $\Lambda$ (resp. $L$) of $X^1-X^2$ (resp. $W$), the equation $\Lambda(t,x)=L(T_t,x)$ holds and in particular $\Lambda(t,0)=L(T_t,0)$. Temporarily we use the notation $\Lambda(t)$ (resp. $L(t)$) to denote $\Lambda(t,0)$ (resp. $L(t,0)$) in the left of this paragraph, and define the left inverse $L^{-1}(u):=\inf\{t:L(t)>u\}$ (same for $\Lambda^{-1}(u)$). Then $L\big(L^{-1}(u)\big)=u$ and $L(t)\leq u$ if and only if $L^{-1}(u)\geq t$. Since $\Lambda(t)=L\big(T(t)\big)$ and $T_t=A^{-1}_t$ is a continuous and strictly increasing function where $A_t=t+2\nu L(t)$,
  \begin{equation}
  \Lambda^{-1}(u)=T^{-1}\big(L^{-1}(u)\big)=A\big(L^{-1}(u)\big)=L^{-1}(u)+2\nu L\big(L^{-1}(u)\big)= L^{-1}(u) + 2\nu u.
  \end{equation}
  Thereby we have
  \begin{equation}
  \mathbb{P}(\Lambda(t)\leq u)=\mathbb{P}(\Lambda^{-1}(u)\geq t)=\mathbb{P}(L^{-1}(u)\geq t-2\nu u),
  \end{equation}
  which is equivalent to
  \begin{equation}\label{2.3A}
  \mathbb{P}(\Lambda(t,0)> u)= \left\{
  \begin{array}{lr}
  \mathbb{P}(L(t-2\nu u,0)> u), & u\leq \frac{t}{2\nu u}; \\
  0, & u> \frac{t}{2\nu u}.
  \end{array}
  \right.
  \end{equation}
  The distribution of $L(t,0)$ is standard, which is given by
  \begin{eqnarray}\label{2.3B}
  && \mathbb{P}(L(t,0)=0)=2 \Phi\Big(\frac{|x_1-x_2|}{2\sqrt{t}}\Big)-1,\\
  \label{2.3C}&&\mathbb{P}(L(t,0)>u)=2-2\Phi\Big(\frac{|x_1-x_2|+2u}{2\sqrt{t}}\Big), ~~~~{\rm for~all~}u\geq 0.
  \end{eqnarray}
  In fact, since $L(t,0)$ is the solution of a Skorohod equation (see Lemma 6.14 in \cite[Chapter 3]{I.Kara.S.E.Shre.1991}), the distribution of $L$ can be easily derived by solving the Skorohod equation.
  
  We then consider the expectation $\mathbb{E}\left[\Lambda(t,0)\right]$.
  When $x_1-x_2=0$, for $t>0$,
    \begin{align}
    \mathbb{P}(\Lambda_0(t,0)=0)&=\mathbb{P}(L_0(t,0)=0)=0 \\
    \label{2B}
    \mathbb{P}(\Lambda_0(t,0)>u)&= \left\{
    \begin{array}{lr}
    \mathbb{P}(L_0(t-2\nu u,0)>u)=2-2\Phi(\frac{u}{\sqrt{t-2\nu u}}),& t>2\nu u>0; \\
    0, &  2\nu u \geq t>0.
    \end{array}
    \right.
    \end{align}
  Therefore,
  \begin{equation} \label{meanlocaltime}
  \mathbb{E}\left[\Lambda_0(t,0)\right]=\int_{0}^{\infty}\mathbb{P}(\Lambda_0(t,0)>u)du=2\int_{0}^{\frac{t}{2\nu}} \int_{\frac{u}{\sqrt{t-2\nu u}}}^{\infty}\phi(y)dydu,
  \end{equation}
  where $\phi(y):=\frac{1}{\sqrt{2\pi}}e^{-\frac{y^2}{2}}$ is the Gaussian density. For the right-hand side of (\ref{meanlocaltime}), change the order of the integrals, use the substitution $ z=\sqrt{\nu ^2y^2+t}/\nu$ and apply integration by parts,
  \begin{eqnarray} \label{meanlocaltime'}
 \nonumber \mathbb{E}\left[\Lambda_0(t,0)\right] &
   = & 2\int_{0}^{\infty} (y\sqrt{\nu^2y^2+t}-\nu y^2) \phi(y)dy \\
 \nonumber& = & -\frac{2}{\sqrt{2\pi}}\int_{\frac{\sqrt{t}}{\nu}}^{\infty} \nu z~ de^{-\frac{z^2}{2}+\frac{t}{2\nu^2}} - 2\nu\int_{0}^{\infty}y^2\phi(y)dy \\
 \nonumber& = & \sqrt{2t/\pi} + 2\nu\int_{\sqrt{t}/\nu}^{\infty} \frac{1}{\sqrt{2\pi}}e^{-\frac{z^2}{2}+\frac{t}{2\nu^2}}dz - \nu \\
 & = & \sqrt{2/\pi}~ t^{1/2} + \left(2 e^{t/2\nu^2}\left[1-\Phi\big(\sqrt{t}/\nu\big)\right] - 1\right)\nu.
 \end{eqnarray}
 Note that $e^{x^2/2}\left(1-\Phi(x)\right)$ is a decreasing function when $x\geq 0$, so the term in the bracket of the last line in (\ref{meanlocaltime'}) is bounded by 3. Consequently, the order of $\mathbb{E}\left[\Lambda_0(t,0)\right]$ is $t^{1/2}$.
 
 Generally when $x_1-x_2\neq 0$, according to (\ref{2.3A})-(\ref{2.3C}), for all $u\geq0$,
   \begin{equation}\label{2H}
   \mathbb{P}\big(\Lambda_{x_1-x_2}(t,0)>u\big)\leq \mathbb{P}\big(\Lambda_0(t,0)>u\big).
   \end{equation}
 Hence,
 \begin{equation} \label{estimate_mean_local_time}
 \mathbb{E}[\Lambda_{x_1-x_2}(t,0)]\leq \mathbb{E}[\Lambda_0(t,0)]= O(t^{1/2}).
 \end{equation}
 This completes the proof.
 \end{proof}
 
 \begin{lem} \label{lem-guji2}
 Let $x_1-x_2=0$ and $\tilde{W}_t:=X_t^1-X_t^2$. For a fixed $t$, let $E_k:=\big\{\tilde{W}_s=0:$ {\rm for some} $s\in\big(kt,(k+1)t\big]\big\}$. Then for any $\alpha>0$,
 \begin{equation} \label{condAl}
 \sum_{k=0}^{n} \mathbb{P}(E_k)=o(n^{1/2+\alpha}),
 \end{equation}
 where $f(n)=o\big(g(n)\big)$ denotes that $f(n)/g(n)\rightarrow 0$ as $n\rightarrow\infty$. 
 \end{lem}
 \begin{proof}
 $\tilde{W}_t$ can be obtained by time-changing a Brownian motion $W_t$ as in Lemma \ref{lem-guji1}. Following the notations there, we write $A_t=t+2\nu L_0(t,0)$, where $L_0(t,x)$ is the local time of $W_t$, and $T_t=A_t^{-1}$ so that $\tilde{W}_t=W_{T_t}$. Now we define a measure $m(dx):=dx+2\nu 1_{\{0\}}(x)$ (in some references this measure is called speed measure, see \cite{MR0297016}), then by Proposition \ref{localtime} (ii), (iv) and $\mathbb{E}[W_t^2]=2t$, for any bounded measurable function $f$, we have a.s. 
 \begin{eqnarray}
 \nonumber\int_{0}^{t} f(W_{T_s})ds &=& \int_{0}^{T_t} f(W_s)dA_s 
= \int_{0}^{T_t} f(W_s)ds +2\nu \int_{0}^{T_t} f(W_s)L_0(ds,0) \\
 \nonumber&=& \int_{-\infty}^{\infty}f(x) L_0(T_t,x)dx +2\nu f(0)L_0(T_t,0) \\
 &=& \int_{-\infty}^{\infty}f(x) L_0(T_t,x)m(dx)
 \end{eqnarray}
 Since $L_0(T_t,x)=\Lambda_0(t,x)$, where $\Lambda_0(t,x)$ is the local time of $\tilde{W}$, this equality is equivalent to
 \begin{equation}
 \int_{0}^{t} f(\tilde{W}_s)ds = \int_{-\infty}^{\infty}f(x) \Lambda_0(t,x)m(dx)
 \end{equation}
 Taking differentiation of both sides with respect to $t$ and then taking the expectation gives us the probability density $p_t(x)$ of $\tilde{W}_t$ with respect to $m(dx)$, 
 \begin{eqnarray}
 p_t(x)= \frac{\partial \mathbb{E}\Lambda_0(t,x)}{\partial t}.
 \end{eqnarray}
 From the expression (\ref{2.13}), $p_t(0)\leq Ct^{-1/2}$ for some constant $C$ when $t$ is large enough. Note that $\tilde{W}$ behaves as a Brownian motion when it is not at $0$, and $m$ is Lebesgue measure on $\mathbb{R}\backslash\{0\}$. Consequently, $p_t(x)\leq Ct^{-1/2}$ also holds for $x\neq 0$. Hence for any $K>0$, 
 \begin{equation} \label{2.2A}
 \mathbb{P}(|\tilde{W}_t|\leq K)=\int_{-K}^{K}p_t(x)m(dx)\leq (K+2\nu)Ct^{-1/2}
 \end{equation}
 For the event $E_k$ and any $\alpha>0$,
 \begin{eqnarray}\label{2.29}
 \mathbb{P}(E_k) &\leq& \mathbb{P}(E_k\cap \{|\tilde{W}_{kt}|>(kt)^{\alpha} \})+\mathbb{P}(\{|\tilde{W}_{kt}|\leq(kt)^{\alpha} \}).
 \end{eqnarray}
 Conditional on $\tilde{W}_{kt}=x>0$, the probability of $\tilde{W}_{s}$ hitting $0$ on $[kt,(k+1)t]$ is the same as the one of a Brownian motion starting from $x$ hitting $0$, and it decreases as $|x|$ increases. Therefore, recall the definition of $E_k$,
 \begin{eqnarray}
 \mathbb{P}(E_k\cap \{|\tilde{W}_{kt}|>(kt)^{\alpha} \}) \leq 2\left(1-\Phi\Big(\frac{(kt)^{\alpha}}{\sqrt{2t}}\Big)\right)
 \end{eqnarray}
 where the right-hand side is the probability of a Brownian motion starting from $(kt)^{\alpha}$ hitting zero before time $t$. Since $1-\Phi(x)$ has a Gaussian decay, 
 \begin{equation}\label{2.2B}
 \mathbb{P}(E_k\cap \{|\tilde{W}_{kt}|>(kt)^{\alpha} \}) \leq C k^{-1/2+\alpha}.
 \end{equation}
 Applying (\ref{2.2A}) and (\ref{2.2B}) to (\ref{2.29}), we have
 \begin{eqnarray} \label{2.35}
 \mathbb{P}(E_k)=O(k^{-1/2+\alpha}).
 \end{eqnarray}
 Since $\alpha>0$ is arbitrary, taking summation of (\ref{2.35}) from $1$ to $n$ gives the desired result.
 \end{proof}
 
 \end{subsection}
 \begin{subsection}{Variance of the 2-point motion}
 In this subsection, we will compute the covariance of the Howitt-Warren 2-point motion. The following lemma and remark will be applied in Section 3 and 4.
 \begin{lem}\label{variance}
 Let $(X^1_t,X^2_t)$ be a Howitt-Warren 2-point motion starting from $(x_1,x_2)$ with stickiness parameter $\nu$. Then the covariance can be expressed as
 \begin{equation}\label{2.33}
 Cov(X^1_t,X^2_t)=G(x_1-x_2,t)+H(x_1-x_2,t),
 \end{equation}
 where
  \begin{equation}\label{5D}
  G(x,t):=\sqrt{2}\nu\int_{0}^{2t/x^2} \frac{\sqrt{2t-x^2s}} {\pi s^{3/2}} e^{-1/2s}ds,
  \end{equation}
  \begin{equation}\label{5E}
  H(x,t):=2\nu^2\int_{0}^{t}\bigg\{2 e^{(t-s)/2\nu^2}\Big[1-\Phi\big(\sqrt{t-s}/\nu\big)\Big] - 1\bigg\} \frac{\partial \Psi}{\partial s}(x_1-x_2,s) ds,
  \end{equation}
  where $\Psi(x,s)=2-2\Phi\big(|x|/\sqrt{2s}\big)$, and $\Phi$ is the standard Gaussian distribution function.
  Moreover, {\rm(i)} $\frac{\partial G}{\partial x}$ and $H$ are uniformly bounded by $2\nu$ and $6\nu^2$ respectively, {\rm(ii)} $G(\sqrt{n}x,nt)=\sqrt{n}G(x,t)$, and {\rm(iii)} for the covariance function $\Gamma$ as given in $($\ref{CovFun}$)$, $\Gamma\big((t,x_1),(t,x_2)\big)=G(x_1-x_2,t)$.
 \end{lem}
 \begin{proof}
 Assume the drift $\beta=0$ without loss of generality. Note that by Definition \ref{2pointmotion}, the covariance is given by
 \begin{equation}
 Cov(X^1_t,X^2_t)=\int_{0}^{t}1_{\{X^1_s=X^2_s\}}ds=\nu\big(E|X^1_t-X^2_t|-|x_1-x_2| \big). \end{equation}
 In order to show (\ref{2.33}), one only need to compute the first moment of the sticky Brownian motion $\tilde{W}_t=X^1_t-X^2_t$.
 Let $\tau$ be the stopping time of $\tilde{W}_t$ first hitting the origin. Since before $\tau$, $|\tilde{W}_t|$ behaves as a Brownian motion $W_t$ with quadratic variation $2t$ and starting point $|x_1-x_2|$, by the reflection principle,
 \begin{eqnarray}\label{2.38}
 P(\tau>t)=\mathbb{P}(\inf_{0\leq s\leq t}W_s>0)=2-2\Phi\Big(\frac{|x_1-x_2|}{\sqrt{2t}}\Big)=\Psi(x_1-x_2,t).
 \end{eqnarray}
 Let $\sigma$ be the stopping time of $W_t$ first hitting the origin. Note that on the event $\{\sigma\leq t\}$, the expectation of $W_t$ is $0$.
 \begin{eqnarray}
 E\big[|\tilde{W}_t|1_{\{\tau>t\}}\big]=\mathbb{E}\big[W_t1_{\{\sigma>t\}}\big]=\mathbb{E}\big[W_t\big]-\mathbb{E}\big[W_t1_{\{\sigma\leq t\}}\big]=|x_1-x_2|.
 \end{eqnarray}
 This last calculation implies that 
 \begin{equation}\label{2.40}
 E\big|\tilde{W}_t\big| -\left|x_1-x_2\right|=E\big[|\tilde{W}_t|1_ {\{\tau\leq t\}}\big].
 \end{equation}
 Conditional on the stopping time $\tau$, $\tilde{W}_t$ is a sticky Brownian motion starting from 0. Consider the first moment of a sticky Brownian motion $\tilde{W}^0_s$ starting form the origin with stickiness also at the origin. By the Tanaka-Meyer formula and Lemma \ref{lem-guji1},
  \begin{eqnarray}\label{2.41}
  \mathbb{E}\big|\tilde{W}^0_s\big|=2\mathbb{E}\big[\Lambda_0(s,0)\big]
  = 2\sqrt{2/\pi}~ s^{1/2} + 2\nu\left(2 e^{s/2\nu^2}\Big[1-\Phi\big(\sqrt{s}/{\nu}\big)\Big] - 1\right).
  \end{eqnarray}
  If we condition $\tilde{W}_t$ on $\tau$ and use the strong Markov property, then by (\ref{2.38}), (\ref{2.40}) and (\ref{2.41}),
  \begin{eqnarray}
  \nonumber Cov(X_t^1,X_t^2)&=&\nu \int_{0}^{t} \bigg\{ 2\sqrt{2(t-s)/\pi} + 2\nu\Big(2 e^{(t-s)/2\nu^2}\Big[1-\Phi\big(\sqrt{t-s}/{\nu}\big)\Big] - 1\Big) \bigg\} P(\tau\in ds) \\
  &=& \sqrt{2}\nu\int_{0}^{\frac{2t}{(x_1-x_2)^2}} \frac{\sqrt{2t-(x_1-x_2)^2s}}{\pi s^{3/2
  }}e^{-1/2s}ds \\
  &&+2\nu^2\int_{0}^{t}\bigg\{2 e^{(t-s)/2\nu^2}\Big[1-\Phi\big(\sqrt{t-s}/\nu\big)\Big] - 1\bigg\} P(\tau\in ds)\\
  \nonumber &=& G(x_1-x_2,t)+H(x_1-x_2,t).
  \end{eqnarray}
  
  As for the statement (i), since $e^{x^2/2}(1-\Phi(x))$ is bounded by $1$ when $x\geq 0$, for all $x$ and $t$,
  \begin{eqnarray}
  &\left|H(x,t)\right| \leq 2\nu^2\int_{0}^{t}(2+1)P(\tau\in ds) \leq 6\nu^2, \\
  & \Big|\frac{\partial G}{\partial x}(x,t)\Big| = \sqrt{2}\nu\Big| \int_{0}^{2t/x^2} \frac{-xs}{\pi s^{3/2}\sqrt{2t-x^2s}}e^{-1/2s}ds \Big| \leq \frac{2\nu}{\pi} \int_{0}^{2t/x^2}\frac{|x|}{\sqrt{s(2t-x^2s)}}ds\leq2\nu.
  \end{eqnarray}
  (ii) follows directly form the expression of $G$. By basic calculus, $\Gamma\big((t,0),(t,0)\big)=\nu\sqrt{t}=G(0,t)$, $\frac{\partial \Gamma((t,x),(t,0))}{\partial x}\big|_{x=0}\!=\!\frac{\partial G(x,t)}{\partial x}\big|_{x=0}$, $\frac{\partial^2 \Gamma((t,x),(t,0))}{\partial x^2}\!=\!\frac{\partial^2 G(x,t)}{\partial x^2}$, and $\Gamma\big((t,x_1),(t,x_2)\big)\!=\!\Gamma\big((t,x_1-x_2),(t,0)\big)$. Hence the identity in (iii) holds. At first sight, this identity is not trivial, so we also provide a probabilistic explanation in Appendix B. 
 \end{proof}
 \begin{remark}\label{rmk2.8}
 \rm Note that the Howitt-warren flow has independent increments, i.e., the random environment on disjoint time intervals are independent. Therefore, one can easily modify the above proof to obtain the following conditional covariance of two random motions in the Howitt-Warren flow: for $t>s$,
 \begin{equation}\label{conditional_variance}
 E\Big[\langle X^1,X^2\rangle(t)-\langle X^1,X^2\rangle(s) \Big|(X_s^1,X_s^2)\Big]=G(X_s^1-X_s^2,t-s)+H(X_s^1-X_s^2,t-s),
 \end{equation}
 where $\langle X^1,X^2\rangle$ denotes the covariation process.
 (\ref{conditional_variance}) will be used in the proof of Lemma \ref{one-time-level}.
 \end{remark}
 
 \end{subsection}
 
 \end{section}

 \begin{section}{Proof of Theorem \ref{qIP}}
 The proof of Theorem \ref{qIP} is based on a second moment method, which is inspired by \cite{M.Birk.J.CernyA.Depp.N.Gant.2013}. The idea is that the averaged law of a random motion $(X_t)_{t\geq 0}$ in the Howitt-Warren flow converges to the law of a drifted Brownian motion, and the quenched law satisfies a law of large numbers by variance calculations so that it also converges to the same limit. In the proof, we consider a class of test functions applied to $Y^{(n)}:=(Y_{nt}/\sqrt{n})_{0\leq t\leq T}$ for some fixed $T>0$, where $Y_t:=X_t-\beta t$, and bound the variance of their quenched mean $E^{\omega}f(Y^{(n)})$ in such a way that we can apply the Borel-Cantelli lemma to get an almost sure convergence for a subsequence of the quenched mean. Then with some modification we will reach our goal. This method can also be applied to show the quenched invariance principle for the random walk in an i.i.d. space-time random environment (see Appendix A). We begin with two lemmas.
 
 \begin{lem} \label{3A}
 Let $T>0$, and $C[0,T]$ be the space of continuous function on $[0,T]$ equipped with the sup norm $\parallel\cdot\parallel$. If for any bounded Lipschitz function $f:C[0,T]\rightarrow\mathbb{R}$, $E^{\omega}\big[f(Y^{(n)})\big]$ converges to $E\big[f(B)\big]$ a.s. as $t$ tends to $\infty$, where $B:=(B_t)_{0\leq t\leq T}$ is a standard Brownian motion, then for $\mathbb{P}-a.e.$ $\omega$, $Y^{(n)}$ converges weakly to $B$.
 \end{lem}
 \begin{proof}
 It suffices to show that there exists a convergence determining class for $C[0,T]$ that consists of countably many bounded Lipschitz functions. However, the proof of Proposition 3.17 in \cite{MR900810} shows how to find such a convergence determining class for general Polish space. As a particular case Lemma \ref{3A} holds.
 \end{proof}
 To check the almost sure convergence for bounded Lipschitz function $f$, we first consider its variance.
 
 \begin{lem}\label{3B}
 For any bounded Lipschitz function $f$ and a random motion $(X_t)_{t\geq 0}$ with $X_0=0$ in the Howitt-Warren flow, there exists a constant $C_{f,T,\nu}>0$, depending only on $f$, $T$, and the stickiness parameter $\nu$ of the Howitt-Warren $2$-point motion, such that
 \begin{equation}
 \mathbb{E}\Big[\big(E^{\omega}f(Y^{(n)})-Ef(B)\big)^2\Big] \leq C_{f,T,\nu} n^{-1/4}.
 \end{equation}	
 \end{lem}
 \begin{proof}
 Let $X_t^1$, $X_t^2$ with $X^1_0=X^2_0=0$ be two independent random motions in the same environment, i.e., for a fixed realization of the Howitt-Warren flow. Then under the averaged law $P$, $(X_t^1,X_t^2)$ is a $2$-point motion. Moreover, by Theorem \ref{thmsde} we have a coupling $(X_t^1,X_t^2,B_t^1,B_t^2,B_t^3)$ as in (\ref{spde}), where $B_t^1$,$B_t^2$ and $B_t^3$ are independent Brownian motions. Let $Y_t^i:=X_t^i-\beta t$, $Y^{i,(n)}:=(Y^i_{nt}/\sqrt{n})_{0\leq t\leq T}$ for i=1,2, and $B^{j,(n)}:=(B^j_{nt}/\sqrt{n})_{0\leq t\leq T}$ for $j=1,2,3$.
 \begin{eqnarray} \label{3.1}
 \nonumber&& E\big[\parallel Y^{1,(n)}-B^{1,(n)}\parallel \big] = n^{-\frac{1}{2}}E\big[\sup_{0\leq t\leq T}\big|\int_{0}^{nt}1_{\{X_t^1=X_t^2\}}dB_s^3-\int_{0}^{nt}1_{\{X_t^1=X_t^2\}} dB_s^1\big| \big] \\
 \nonumber&&~\leq~  4n^{-\frac{1}{2}}\Big\{E\Big[\Big(\int_{0}^{nT}1_{\{X_t^1=X_t^2\}}dB_s^3\Big) ^2\Big]\Big\}^{\frac{1}{2}}\!\! +\! 4n^{-\frac{1}{2}}\Big\{E\Big[\Big(\int_{0}^{nT}1_ {\{X_t^1=X_t^2\}}dB_s^1\Big)^2\Big]\Big\}^{\frac{1}{2}} \\
 &&~= ~ 8n^{-\frac{1}{2}}\Big(E\Big[\int_{0}^{nT}1_{\{X_t^1=X_t^2\}}ds \Big]\Big)^\frac{1}{2}~ = ~ 16n^{-\frac{1}{2}}\nu \Big(E\left[\Lambda_0(nT,0) \right]\Big)^\frac{1}{2}~\leq~  C_{T,\nu} n^{-\frac{1}{4}},
 \end{eqnarray}
 where the first equality holds because of the coupling (\ref{spde}), the second step is by Doob's $L_2$ inequality, and the last two steps are by (\ref{spde}) and (\ref{2.13}). Since $X_t^1$, $X_t^2$ are independent in a same environment $\omega$ and evolve as a $2$-point motion under the averaged law $P$, we have
 \begin{eqnarray}
 \nonumber & & \mathbb{E}\big[\big(E^{\omega}f(Y^{(n)})-Ef(B)\big) ^2 \big] \\
 \nonumber &= & E \big[f(Y^{1,(n)})f(Y^{2,(n)})\big] - E\big[f(B^{1,(n)})f(B^{2,(n)})\big] \\
 \nonumber &\leq & \parallel f \parallel_{\infty} C_f \left( E\big[\parallel Y^{1,(n)}-B^{1,(n)}\parallel \big] + E\big[\parallel Y^{2,(n)}-B^{2,(n)}\parallel \big] \right) \\
 &\leq  & C_{f,T,\nu} n^{-1/4},
 \end{eqnarray}
 where the first inequality holds because $f$ is Lipschitz with Lipschitz constant  $C_f$. 
 \end{proof}
 Now we can finish the proof of Theorem \ref{qIP}.
 \begin{proof}[Proof of Theorem \ref{qIP}] 
 For the first statement, we only need to prove that for all $T>0$, $(\frac{Y_{nt}}{\sqrt{n}}) _{0\leq t\leq T}$ converges weakly to $(B_t)_{0\leq t\leq T}$ in $C[0,T]$ equipped with the sup norm. We follow the notations given in Lemma \ref{3A} and Lemma \ref{3B}. By Lemma \ref{3A}, it only remains to show that for each bounded Lipschitz function $f$, a.s., $E^{\omega}f(Y^{(n)})\rightarrow Ef(B)$. Without loss of generality, we assume $Y_0=0$.
 
 First, observe that along a subsequence $k_n=n^5$, the almost sure convergence holds. Indeed, for any $\epsilon>0$, by Lemma \ref{3B} and the Markov inequality,
 \begin{equation}
 \mathbb{P}\left(\big|E^{\omega}f(Y^{(n^5)})-Ef(B)\big|\geq \epsilon \right)\leq C_{f,T,\nu} \epsilon^{-2}n^{-5/4}.
 \end{equation}
 It is summable, and hence by the Borel-Cantelli lemma $E^{\omega}f(Y^{(n^5)})\rightarrow Ef(B)$ a.s..
 
 For general $m\in[n^5,(n+1)^5)$,  we bound the maximum of the differences between $E^{\omega}f(Y^{(m)})$ and $E^{\omega}f(Y^{(n^5)})$. Since $(n+1)^5-n^5\leq 6n^4$ for large $n$,
 \begin{eqnarray}
 \nonumber & & \max_{n^5\leq m<(n+1)^5} \Big|E^{\omega}f(Y^{(m)})-E^{\omega} f(Y^{(n^5)})\Big| \\
 \nonumber & \leq & \max_{n^5\leq m<(n+1)^5} C_f \Big\{
 \Big(\sqrt{\frac{m}{n^5}}-1\Big)E^{\omega}\parallel Y^{(m)}\parallel+ 
 E^{\omega}\Big[n^{-5/2}\sup_{0\leq t\leq T} \big|Y_{mt}-Y_{n^5t}\big|\Big]\Big\} \\
 & \leq & C'_f n^{-1}\max_{n^5\leq m<(n+1)^5}E^{\omega}\parallel Y^{(m)}\parallel + C''_f n^{-5/2}E^{\omega}\Big[\sup_{0\leq t\leq T}\sup_{0\leq s<6n^4} \big|Y_{n^5t+sT}-Y_{n^5t}\big|\Big].
 \end{eqnarray}
 For the process $M_t:=\sup_{0\leq s<6n^4} \big|Y_{n^5t+sT}-Y_{n^5t}\big|$, we note that for any $s\leq t\leq s+6n^4T$, $M_t\leq M_s+M_{s+6n^4T}$.
 Therefore, setting $T_n:=\{6Tk/n:0\leq k\leq [n/6]+1 \}$ gives the following bound:
 \begin{equation}
 \sup_{0\leq t\leq T}M_t\leq 2\max_{t\in T_n}M_t
 \end{equation}
 Moreover, since $Y_t=X_t-\beta t$ is a standard Brownian motion starting form $0$ under the averaged law $P$, the $k$-th moment of $\sup_{0\leq s\leq t}|Y_s|$ is of order $t^{k/2}$ and $M_t$ has stationary increments. Hence by the Markov inequality, for any $\delta>0$,
 \begin{equation} \label{3C}
 \mathbb{P}\Big(n^{-1}\max_{n^5\leq m<(n+1)^5}E^{\omega}\parallel Y^{(m)}\parallel>\delta \Big)\leq \delta^{-2} n^{-7} E\Big[\max_{0\leq m<(n+1)^5T}|Y_t|^2 \Big]= O(n^{-2}),
 \end{equation}
 \begin{eqnarray}
 \label{3D} \mathbb{P}\Big(n^{-5/2}E^{\omega}\big[\sup_{0\leq t\leq T}M_t\big]>\delta\Big) \leq 
  2^6\delta^{-6} n^{-15} E\Big[\max_{t\in T_n}M_t^6\Big]\leq2^6\delta^{-6} n^{-15} nE\Big[M_0^6\Big] = O(n^{-2}),
 \end{eqnarray}
 where the last inequality in (\ref{3D}) holds because $E\big[\max_{t\in T_n}M_t^6\big]\leq E\big[\sum_{t\in T_n}M_t^6\big] \leq  nE\big[M_0^6\big]$.
 Both (\ref{3C}) and (\ref{3D}) are summable, which implies that for $\mathbb{P}$-a.e. $\omega$, as $n\rightarrow\infty$,
 \begin{equation}
 \max_{n^5\leq m<(n+1)^5} \left|E^{\omega}f(Y^{(m)})-E^{\omega} f(Y^{(n^5)})\right| \longrightarrow 0.
 \end{equation}
 
 Therefore, for each bounded Lipschitz function $f$, $\mathbb{P}$-a.s., $E^{\omega}\big[f(Y^{n})\big] \rightarrow Ef\big[(B)\big]$, which finishes the proof of the first part.
 
 To show that for $\mathbb{P}$-a.e. $\omega$, $Z_n:=n^{-1/2}\max_{s\leq nt}\big|E^{\omega}X_s-\beta s\big|$ converges to $0$, we again consider the second moment. By Doob's $L_2$ inequality and Lemma \ref{variance}
 \begin{eqnarray} \label{3AAA}
 \mathbb{E}\left[\left( n^{-1/2}\max_{s\leq nt}\big|E^{\omega}X_s-\beta s\big|\right)^2\right] \leq 4n^{-1}E[Y^1_{nt}Y^2_{nt}] = O(n^{-1/2}),
 \end{eqnarray}
 Therefore, the order of the second moment of $Z_n$ is $n^{-1/2}$. So along the subsequence $\{n^3\}$, $Z_{n^3}\rightarrow 0$ by the Markov inequality and Borel-Cantelli. As for $n^3\leq k<(n+1^3)$,
 \begin{equation}
 \big|Z_k-Z_{n^3}\big|\leq C_1 n^{-5/2}\max_{s\leq n^3t}\big|E^{\omega}X_s-\beta s\big|+C_2 n^{-3/2}\max_{n^3t\leq s\leq (n+1)^3t}\big|E^{\omega}[X_s-X_{n^3t}]-\beta (s-n^3t)\big|.
 \end{equation}
 Similar argument as (\ref{3C}) and (\ref{3D}) can be applied here, and therefore we can show that for $\mathbb{P}$-a.e. $\omega$, $\max\limits_{n^3\leq k<(n+1^3)}\big|Z_k-Z_{n^3}\big|$ converges to $0$. Thus for $\mathbb{P}$-a.e. $\omega$, $Z_n$ converges to $0$. 
 \end{proof}
 
 \end{section}
 
 \begin{section}{Proof of Theorem \ref{GaussianlimitThm}}
 In this section we will only focus on $a_n(t,r)$ as defined in Theorem \ref{GaussianlimitThm}, since the proof for $b_n(t,r)$ can be easily obtained by a translation of the environment, see Remark \ref{rmk1.7}.
 
 The strategy is essentially the same as in \cite{M.Bala.F.Rass.T.Sepp.2006}. We proceed in two steps. First it will be shown in Lemma \ref{one-time-level} that for a fixed time $t$, the distribution of $(a_n(t,r_1),\cdots, a_n(t,r_k))$ for any given integer $k>0$ converges weakly to $(a(t,r_1),\cdots, a(t,r_k))$. 
 
 In the second step, observe that by decomposing $X^{r\sqrt{n}-\beta nt,-nt}_0$ in terms of its increments on $[-nt,-n(t-s)]$ and $[-n(t-s),0]$,
  \begin{eqnarray}
  \nonumber a_n(t,r)&=& n^{-1/4}\left(E^{\omega}X^{r\sqrt{n}-\beta nt,-nt}_0-r\sqrt{n}\right)\\
  \nonumber&=& n^{-1/4}\left(E^{\omega}X^{r\sqrt{n}-\beta nt,-nt}_0-E^{\omega}X^{r\sqrt{n}-\beta nt,-nt}_{-n(t-s)} \right)+n^{-1/4}\left( E^{\omega}X^{r\sqrt{n}-\beta nt,-nt}_{-n(t-s)}-r\sqrt{n}\right) \\
  \nonumber&=& n^{-1/4}\int \left(E^{\omega}X^{z-\beta ns,-ns}_0-z\right) P^{\omega}\left(X^{r\sqrt{n}-\beta nt,-nt}_{-n(t-s)}+\beta ns\in dz\right) \\
  \nonumber &&+n^{-1/4} \left(E^{\omega}X^{r\sqrt{n}-\beta nt,-nt}_{-n(t-s)}-r\sqrt{n}\right) \\
  \label{5G}&=& \int a_n(s,\frac{z}{\sqrt{n}}) P^{\omega}\left(X^{r\sqrt{n}-\beta nt,-nt}_{-n(t-s)}+\beta ns\in dz\right) + a_n(t-s,r)\circ T_{-ns,-\beta ns},
  \end{eqnarray}
  where $T_{t,x}$ denotes the translation of the random environment that makes (x,t) the new space-time origin. Here we note that in the decomposition, as functions of the random environment in disjoint time intervals, $a_n(s,\frac{z}{\sqrt{n}})$ and $a_n(t-s,r)\circ T_{-ns,-\beta ns}$ are independent, while the random measure in the right-hand side of (\ref{5G}) converge to Gaussian distribution by Theorem \ref{qIP}.
  
  Meanwhile, for the limiting Gaussian process $\big(a(t,r)\big)$ in Theorem \ref{GaussianlimitThm}, \cite[Lemma 3.1]{M.Bala.F.Rass.T.Sepp.2006} shows that
  \begin{prop} \label{prop1}
  There is a version of the Gaussian process $\{a(t,r):(t,r)\in\mathbb{R}^+\times\mathbb{R}\}$ that is continuous in $(t,r)$. Moreover, given $0=t_0<t_1<\cdots<t_n$, let $\{\tilde{a}(t_i-t_{i-1},\cdot):1\leq i\leq n\}$ be independent random functions such that $\tilde{a}(t_i-t_{i-1},\cdot)$ has the distribution of $a(t_i-t_{i-1},\cdot)$ for all $i$. Define $a^*(t_1,r)=\tilde{a}(t_1,r)$ for $r\in\mathbb{R}$ and inductively for $i=2,\cdots,n$ and $r\in\mathbb{R}$,
  \begin{equation}\label{decomp.a}
  a^*(t_i,r):=\int a^*(t_{i-1},r+z)\phi\big(\frac{z}{\sqrt{t_i-t_{i-1}}}\big)dz + \tilde{a}(t_i-t_{i-1},r),
  \end{equation}
  where $\phi(x)=\frac{1}{\sqrt{2\pi}}e^{-x^2/2}$.
  Then the joint distribution of the random functions $\{a^*(t_i,\cdot):1\leq i\leq n\}$ is the same as that of $\{a(t_i,\cdot):1\leq i\leq n\}$.
  \end{prop} 
  
  We see that the decompositions (\ref{5G}) and (\ref{decomp.a}) have the same structure. In order to show the convergence of the finite dimensional distributions of $a_n(t,r)$, we only need to take advantage of this structure, and apply induction to $a_n(t,r)$.
  
 \begin{subsection}{A deterministic time-level}
 In this subsection we prove the weak convergence of the finite dimensional distributions of $a_n(t,r)$ for a fixed time $t$ as formulated in the following lemma:
 \begin{lem}\label{one-time-level}
 For any fixed $t>0$ and $N\in\mathbb{N}$, let $r_1<\cdots<r_N$ be $N$ points on the real line. Then the $\mathbb{R}^N$-valued random vector $(a_n(t,r_1),\cdots, a_n(t,r_N))$ converges weakly to the mean zero Gaussian vector $(a(t,r_1),\cdots, a(t,r_N))$ with covariance matrix $\left(\Gamma\big((t,r_i),(t,r_j)\big)\right)_{1\leq i,j\leq N}$, where $\Gamma\big((t,r_i),(t,r_j)\big)=\nu\int_{0}^{2t}\frac{1}{\sqrt{\pi u}}e^{-(r-q)^2/(2u)}du$.
 \end{lem}

 \begin{proof}
  By the Cram\'er-Wold device, we only need to show that for each $\vec{\theta}\in\mathbb{R}^N$, $\sum_{i=1}^N\theta_i a_n(t,r_i)$ converges weakly to $\sum_{i=1}^N\theta_i a(t,r_i)$. Abbreviating \begin{equation}\label{notation4.3}
  X_{s}^{n.i}:=X^{r_i\sqrt{n}-\beta nt,-nt}_{-nt+s} ,
  \end{equation}
  we have the decomposition
  \begin{align}
  \nonumber \sum\limits_{i=1}^N\theta_i a_n(t,r_i) &=n^{-1/4} \sum\limits_{i=1}^N \theta_i \sum\limits_{k=1}^{n}E^{\omega}\big[ X_{kt}^{n.i}-X_{(k-1)t}^{n.i}-\beta t\big] \\
  &= \sum\limits_{k=1}^{n} \Big(n^{-1/4}\sum\limits_{i=1}^N \theta_i E^{\omega}\big[ X_{kt}^{n.i}-X_{(k-1)t}^{n.i}-\beta t\big]\Big),
  \end{align}
  If we define $z_{n,k}:=n^{-1/4}\sum_{i=1}^N \theta_i E^{\omega}\big[ X_{kt}^{n.i}-X_{(k-1)t}^{n.i}-\beta t\big]$ and the filtration
  $$\mathcal{F}_{n,k}:=\sigma\left(K_{u,v}:-nt\leq u\leq v\leq -nt+kt\right)~~~~k=0,1,2,\cdots,$$
  where $(K_{s,t})_{s\leq t}$ is the underlying Howitt-Warren flow,
  then $\sum_{i=1}^N\theta_i a_n(t,r_i)=\sum_{k=1}^{n}z_{n,k}$, and the process $(\sum_{k=1}^{j}z_{n,k})_{j\in\mathbb{N}}$ is adapted to $\{\mathcal{F}_{n,j}\}_{j\in\mathbb{N}}$, and $\mathbb{E}\left[z_{n,k}\big|\mathcal{F}_{n,k-1}\right]=0$. Therefore, $(\sum_{k=1}^j z_{n,k})_{j=0}^n$ is a martingale with respect to the filtration $(\mathcal{F}_{n,j})_{j=1}^n$, where $z_{n,k}$ is the martingale difference.
  
  Recall that our goal is to prove that $\sum_{k=1}^n z_{n,k}$ converges in distribution to a Gaussian. By the martingale central limit theorem \cite[Chapter 7, Theorem 7.3]{MR2722836}, it suffices to check that for any $\epsilon>0$, when $n\rightarrow\infty$,
  \begin{eqnarray}
   \label{mcltcond1}  (i)&\sum\limits_{k=1}^{n} \mathbb{E} \left[ z_{n,k}^2 | \mathcal{F}_{n,k-1} \right] \overset{\mathbb{P}}{\longrightarrow} \sum\limits_{1 \leq i , j \leq N} \theta_{i} \theta_{j} \Gamma\big((t,r_i),(t,r_j)\big), \\
   \label{mcltcond2} (ii)& \sum\limits_{k=1}^{n} \mathbb{E} \left[ z_{n,k}^2 \boldsymbol{1}_{ \{|z_{n,k}| \geq \epsilon \} } | \mathcal{F}_{n,k-1} \right] \overset{\mathbb{P}}{\longrightarrow} 0. 
  \end{eqnarray} 
  
  Note that to show Lindeberg's condition (\ref{mcltcond2}), it suffices to show the Lyapunov condition:
  \begin{equation}
  \sum\limits_{k=1}^{n} \mathbb{E} \left[ z_{n,k}^6 | \mathcal{F}_{n,k-1} \right] \overset{\mathbb{P}}{\longrightarrow} 0 ~~~~~~~~~({\rm as}~n\rightarrow\infty),
  \end{equation}
  which is implied by the following estimate:
  \begin{eqnarray}
  \nonumber\sum\limits_{k=1}^{n} \mathbb{E} \left[ z_{n,k}^6 \right] &=& n^{-3/2} \sum\limits_{k=1}^{n} \mathbb{E} \Big[\Big(\sum_{i=1}^N \theta_i E^{\omega}\big[ X_{kt}^{n.i}-X_{(k-1)t}^{n.i}-\beta t\big]\Big)^6 \Big] \\
  \nonumber&\leq& n^{-3/2} \sum_{i=1}^{N} N^5 \theta_i^6 \sum_{k=1}^{n} E\Big[\left(X_{kt}^{n.i}-X_{(k-1)t}^{n.i}-\beta t\right)^6\Big] \\
  &=& C(N,t,\vec{\theta}) n^{-1/2},
  \end{eqnarray}
  where in the last equality we used that $X_{kt}^{n.i}-X_{(k-1)t}^{n.i}$ is a Brownian motion with drift $\beta$ under the averaged law $P$ and therefore the 6th-moment is a constant depending on $t$.
  
 It only remains to check condition (\ref{mcltcond1}). First note that if given the environment $\omega$ we take independent copies $X^{1,n.i}$, $X^{2,n.i}$ of $X^{n,i}$, then under the averaged law $P$, $(X^{1,n.i},X^{2,n.j})$ is a Howitt-Warren $2$-point motion. Therefore, by Lemma \ref{variance},
 \begin{eqnarray} \label{mean-mart-diff}
 \nonumber \sum\limits_{k=1}^{n}\mathbb{E}\left[z^2_{n,k}\right] &=& n^{-1/2}\sum\limits_{1\leq i,j\leq N}\theta_i\theta_j\mathbb{E}\left[ E^{\omega} [ X_{nt}^{n.i}-r_i\sqrt{n}-\beta nt]E^{\omega} [ X_{nt}^{n.j}-r_i\sqrt{n}-\beta nt]\right] \\
 \nonumber &=& n^{-1/2}\sum\limits_{1\leq i,j\leq N}\theta_i\theta_j Cov(X^{1,n.i}_{nt},X^{2,n.i}_{nt})\\
 \nonumber&=& \sum\limits_{1\leq i,j\leq N}\theta_i\theta_j\Big\{n^{-1/2} G\big(\sqrt{n}(r_i-r_j),nt\big)+n^{-1/2}H\big(\sqrt{n}(r_i-r_j),nt\big) \Big\}\\
 \nonumber&=&\sum\limits_{1\leq i,j\leq N}\theta_i\theta_j\Big\{ \Gamma\big((t,r_i),(t,r_j)\big)+n^{-1/2}H\big(\sqrt{n}(r_i-r_j),nt\big) \Big\}\\
 \label{mean-mart-diff-limit}&\overset{n\rightarrow\infty}{\longrightarrow}& \sum\limits_{1\leq i,j\leq N}\theta_i\theta_j \Gamma\left((t,r_i),(t,r_j)\right).
 \end{eqnarray}

 Hence, to check condition (\ref{mcltcond1}), it suffices to show that as $n$ tends to $\infty$, 
 \begin{equation} \label{4.22}
 \sum_{k=1}^{n} \Big(\mathbb{E}\left[z_{n,k}^2|\mathcal{F}_{n,k-1}\right] -\mathbb{E}\left[z_{n,k}^2\right] \Big) \overset{\mathbb{P}} {\longrightarrow}0.
 \end{equation}
 Actually, we will show convergence in $L^2$. Rewrite the left-hand side of (\ref{4.22}) as
 \begin{eqnarray}\label{5A}
 \nonumber&&\sum\limits_{k=1}^{n}\mathbb{E}\left[z^2_{n,k}\big|\mathcal{F}_{n,k-1}\right]-\sum\limits_{k=1}^{n}\mathbb{E}\left[z^2_{n,k}\right] \\
 \nonumber&=&\sum\limits_{k=2}^{n} \sum\limits_{l=1}^{k-1}\Big( \mathbb{E}\left[z^2_{n,k}\big|\mathcal{F}_{n,l}\right]-\mathbb{E}\left[z^2_{n,k}\big|\mathcal{F}_{n,l-1}\right] \Big) \\
 \nonumber&=&\sum\limits_{l=1}^{n-1} \sum\limits_{k=l+1}^{n}\Big( \mathbb{E}\left[z^2_{n,k}\big|\mathcal{F}_{n,l}\right]-\mathbb{E}\left[z^2_{n,k}\big|\mathcal{F}_{n,l-1}\right] \Big) \\
 &=& \sum\limits_{l=1}^{n-1}R_l,
 \end{eqnarray}
 where 
 \begin{equation}\label{notationR_l}
 R_l:= \sum_{k=l+1}^{n}\Big( \mathbb{E}\left[z^2_{n,k}\big|\mathcal{F}_{n,l}\right]-\mathbb{E}\left[z^2_{n,k}\big|\mathcal{F}_{n,l-1}\right] \Big)
 \end{equation} 
 is a random variable determined by the environment up to time $lt-nt$. Consequently, for $l>l'$, $\mathbb{E}\left[R_lR_{l'}\right]=\mathbb{E}\left[\mathbb{E}\left[R_l\big|\mathcal{F}_{n,l-1}\right]R_{l'}\right]=0$. As a result, the $L^2$ norm of (\ref{5A}) is $\mathbb{E}\left[(\sum_{l=1}^{n-1}R_l)^2\right]= \sum_{l=1}^{n-1}\mathbb{E}\left[R_l^2\right]$.

 Before bounding the second moment of (\ref{5A}), let us carefully explain the meaning of some probability laws, which will be used throughout the following paragraphs. Let $\vec{X}^k$ (resp. $\vec{Y}^k$) be $k$ random motions independent under the quenched law $P^{\omega}$.
 Since the Howitt-Warren flow has independent-increments property, under the law $\mathbb{E}\big[P^{\omega}(\cdot)\big|\mathcal{F}_{n,l}\big]$, $k$ random motions $\vec{X}^k$ evolve in a fixed realization $\omega$ of the random environment before time $lt-nt$, and evolve in the averaged environment after $lt-nt$ (i.e., $\vec{X}^k$ is under the averaged law $P$ and evolves as a Howitt-Warren $k$-point motion). For convenience we introduce 
 \begin{equation}\label{notation1}
 P^{\omega}_l(\cdot):=\mathbb{E}\big[P^{\omega}(\cdot)\big|\mathcal{F}_{n,l}\big].
 \end{equation}
 We then explain the more complicated law
 \begin{equation}\label{notation2}
 P_l\big((\vec{X}^k,\vec{Y}^k)\in\cdot\big):=\mathbb{E}\big[P^{\omega}_l(\vec{X}^k\in\cdot)P^{\omega}_l(\vec{Y}^k\in\cdot)\big].
 \end{equation}
 To understand it, we only need to first couple $\vec{X}^k$ and $\vec{Y}^k$ such that they are independent under law $P^{\omega}_l$, and then average over the random environment $\omega$ before time $lt-nt$. As a result, under law $P_l$, $(\vec{X}^k,\vec{Y}^k)$ is a Howitt-Warren $2k$-point motion up to time $lt-nt$, and then splits into two sets of independent Howitt-Warren $k$-point motions.
 
 In this paragraph let $\vec{X}$ (resp. $\vec{Y}$) be two independent random motions $(X^{1,n,i_1},X^{2,n,i_2})$ (resp. $(Y^{1,n,i_3},X^{2,n,i_4})$), where the later notation is given in (\ref{notation4.3}), and define \begin{equation}
 I_l(\vec{x}):=G\big(x_1-x_2,(n-l)t\big)+H\big(x_1-x_2,(n-l)t\big)
 \end{equation} 
 for short, where $G$ and $H$ are given in (\ref{2.33}). To bound the second moment of $R_l$, we consider the first summation in the right-hand side of (\ref{notationR_l}). By Remark \ref{rmk2.8} and the notation (\ref{notation1}),
 \begin{align}\label{4.15}
 \nonumber & \sum_{k=l+1}^{n} \mathbb{E} \left[z_{n,k}^2|\mathcal{F}_{n,l}\right]\\
 \nonumber &=
 n^{-1/2}\sum_{1\leq i,j\leq N} \theta_i \theta_j \mathbb{E}_{n,l}\left[E^{\omega}[X^{1,n.i}_{nt}-X^{1,n.i}_{lt}-(n-l)\beta t]E^{\omega}[X^{2,n.j}_{nt}-X^{2,n.j}_{lt}-(n-t)\beta t] \right] \\
 \nonumber &= n^{-1/2}\!\!\! \sum_{1\leq i,j\leq N}\!\!\!\ \theta_i \theta_j E^{\omega}_l \left[ G\big(X^{1,n.i}_{lt}-X^{2,n.j}_{lt},(n-l)t\big)+H\big(X^{1,n.i}_{lt}-X^{2,n.j}_{lt},(n-l)t\big)\right] \\
 &=  n^{-1/2} \sum_{1\leq i,j\leq N} \theta_i \theta_j E^{\omega}\big[I_l(\vec{X}_{lt}) \big],
 \end{align}
 where in the second equality we used the independence and the conditional covariance function given in Remark \ref{rmk2.8}, and in the last equality we changed $E^{\omega}_l$ to $E^{\omega}$ because $\vec{X}_{lt}$ only depends on the environment up to time $lt-nt$. Similarly, for the second summation in (\ref{notation2}),
 \begin{equation}\label{4.16}
 \sum_{k=l+1}^{n} \mathbb{E} \big[z_{n,k}^2|\mathcal{F}_{n,l-1}\big] = 
 n^{-1/2}\sum_{1\leq i,j\leq N} \theta_i \theta_j E^{\omega}_{l-1} \Big[ I_l\big(\vec{X}_{lt}\big)\Big].
 \end{equation}
 To consider the second moment of $R_l$, we note that
 \begin{eqnarray}
 \nonumber && \mathbb{E}\Big[\Big(E^{\omega}\big[I(\vec{X}_{lt})\big]-E_{l-1}^{\omega}\big[I(\vec{X}_{lt})\big] \Big)^2\Big] \\
 \nonumber&=& \mathbb{E}\Big[E^{\omega}\big[I(\vec{X}_{lt})\big]E^{\omega}\big[I(\vec{Y}_{lt})\big]\Big]-\mathbb{E}\Big[E^{\omega}_{l-1}\big[I(\vec{X}_{lt})\big]E^{\omega}_{l-1}\big[I(\vec{Y}_{lt})\big]\Big]\\
 \label{4.17}&=& E\big[I(\vec{X}_{lt})I(\vec{Y}_{lt})\big]-E_{l-1}\big[I(\vec{X}_{lt})I(\vec{Y}_{lt})\big],
 \end{eqnarray}
 where in the second equality we used the notation (\ref{notation2}). $(\vec{X},\vec{Y})$ in the first term of (\ref{4.17}) is a Howitt-Warren $4$-point motion. On the other hand, by the explanation below (\ref{notation2}), $(\vec{X},\vec{Y})$ in the second term of (\ref{4.17}) is a Howitt-Warren $4$-point motion before $(l-1-n)t$, and becomes two sets of independent Howitt-Warren $2$-point motions during time interval $[(l-1-n)t,(l-n)t]$. Since before time $(l-1-n)t$, law $P$ and $P_{l-1}$ are equal, we can subtract $I(\vec{X}_{(l-1)t})I(\vec{Y}_{(l-1)t})$ from both expectations of (\ref{4.17}). Therefore, (\ref{4.17}) can be rewritten as
 \begin{eqnarray}
 \label{4.18} &
 E\big[I(\vec{X}_{lt})I(\vec{Y}_{lt})-I(\vec{X}_{(l-1)t})I(\vec{Y}_{(l-1)t})\big]-E_{l-1}\big[I(\vec{X}_{lt})I(\vec{Y}_{lt})-I(\vec{X}_{(l-1)t})I(\vec{Y}_{(l-1)t})\big] \\
 \nonumber&= E\big[\big(I(\vec{X}_{lt})\!-\!I(\vec{X}_{(l-1)t})\big)\big(I(\vec{Y}_{lt})\!-\!I(\vec{Y}_{(l-1)t})\big)\big]\!-\!E_{l-1}\big[\big(I(\vec{X}_{lt})\!-\!I(\vec{X}_{(l-1)t})\big)\big(I(\vec{Y}_{lt})\!-\!I(\vec{Y}_{(l-1)t})\big)\big]
 \end{eqnarray}
 Now let $A_l$ denote the event that for some $s\in \big((l-1-n)t,(l-n)t\big]$, $\{X^{1,n.i_1}_s,X^{2,n.i_2}_s\}\cap\{Y^{1,n.i_3}_s,Y^{2,n.i_4}_s\}\neq\phi$. Notice that on the event $A_l^c$, the Howitt-Warren $4$-point motion evolves in the same way as two sets of independent Howitt-Warren $2$-point motions. Therefore, the averaged law $P$ and law $P_{l-1}$ are equal on the event $A_l^c$ for $(\vec{X}_{lt},\vec{Y}_{lt})$. Hence by (\ref{4.15})-(\ref{4.16}), the right-hand side of (\ref{4.18}) is equal to
 \begin{eqnarray}
 \nonumber  &&\!\!\!\!\!\!\!\!\!\!\!\!\! E\big[1_{A_l}\big(I(\vec{X}_{lt})\!-\!I(\vec{X}_{(l-1)t})\big)\big(I(\vec{Y}_{lt})\!-\!I(\vec{Y}_{(l-1)t})\big)\big]\!-\!E_{l-1}\big[1_{A_l}\big(I(\vec{X}_{lt})\!-\!I(\vec{X}_{(l-1)t})\big)\big(I(\vec{Y}_{lt})\!-\!I(\vec{Y}_{(l-1)t})\big)\big]\\
 \nonumber&&\!\!\!\leq E\big[1_{A_l}\big(I(\vec{X}_{lt})-I(\vec{X}_{(l-1)t})\big)^2\big]+ E\big[1_{A_l}\big(I(\vec{Y}_{lt})-I(\vec{Y}_{(l-1)t})\big)^2\big] + \tilde{E}\big[1_{A_l}\big(I(\vec{X}_{lt})-I(\vec{X}_{(l-1)t})\big)^2\big]\\
 \label{4.28}
 &&~+\tilde{E}\big[1_{A_l}\big(I(\vec{Y}_{lt})-I(\vec{Y}_{(l-1)t})\big)^2\big].
 \end{eqnarray}
 
 Recall that $I_l(\vec{x}):=G\big(x_1-x_2,(n-l)t\big)+H\big(x_1-x_2,(n-l)t\big)$, and by Lemma \ref{variance} (i) $\frac{\partial G}{\partial x}$ and $H$ are uniformly bounded by $2\nu$ and $6\nu^2$ respectively. Thereby,
  \begin{eqnarray}
  \nonumber 
  \big(I(\vec{X}_{lt})-I(\vec{X}_{(l-1)t})\big)^2 &\leq& \Big( 2\nu\big|X^{1,n.i_1}_{lt}-X^{1,n.i_1}_{(l-1)t}\big|+2\nu\big|X^{2,n.i_1}_{lt}-X^{2,n.i_1}_{(l-1)t}\big|+12\nu^2 \Big)^2 \\
  &\leq& 12\nu^2\big|X^{1,n.i_1}_{lt}-X^{1,n.i_1}_{(l-1)t}\big|^2+12\nu^2\big|X^{2,n.i_2}_{lt}-X^{2,n.i_2}_{(l-1)t}\big|^2+C_{\nu}.
  \end{eqnarray}
  Since $X^{1,n,i_1}$ and $X^{2,n,i_2}$ are Brownian motions with drift $\beta$ under both $P$ and $\tilde{P}$,
  \begin{eqnarray}
  \nonumber && E\Big[1_{A_l}\big(I(\vec{X}_l)-I(\vec{X}_{l-1})\big)^2\Big]\\
  \nonumber &\leq& 12\nu^2E\Big[1_{A_l}\big|X^{1,n.i_1}_{lt}-X^{1,n.i_1}_{(l-1)t}\big|^2\Big]+12\nu^2E\Big[1_{A_l}\big|X^{2,n.i_2}_{lt}-X^{2,n.i_2}_{(l-1)t}\big|^2\Big]+C_{\nu}P(A_l) \\
  &\leq& 24\nu^2E\Big[ \big(X^{1,n.i_1}_{lt}-X^{1,n.i_1}_{(l-1)t}\big)^2 1_{|X^{1,n.i_1}_{lt}-X^{1,n.i_1}_{(l-1)t}|> n^{1/6}}\Big] +\big(24\nu^2 n^{1/3}+C_{\nu}\big)P(A_l),
  \end{eqnarray}
  where in the second inequality we decomposed the expectation into two parts: $|X^{1,n.i_1}_{lt}-X^{1,n.i_1}_{(l-1)t}|> n^{1/6}$ and $|X^{1,n.i_1}_{lt}-X^{1,n.i_1}_{(l-1)t}|\leq n^{1/6}$, and in the first part we used $1_{A_l}\leq 1$, while in the second part we bounded the random motion directly by $n^{1/6}$. Since the tail probability of a Brownian motion at a fixed time has a Gaussian decay, when $n$ is large enough, $$E\Big[ \big(X^{1,n.i_1}_{lt}-X^{1,n.i_1}_{(l-1)t}\big)^2 1_{|X^{1,n.i_1}_{lt}-X^{1,n.i_1}_{(l-1)t}|> n^{1/6}}\Big]<n^{-1}.$$
  Hence we have the bound
  \begin{equation}\label{4.30}
  E\Big[1_{A_l}\big(I(\vec{X}_{lt})-I(\vec{X}_{(l-1)t})\big)^2\Big] \leq Cn^{-1}+Cn^{1/3}P(A_l).
  \end{equation}
  With the same argument, this estimate also holds for the other three terms in (\ref{4.28}). Moreover, by Lemma \ref{lem-guji2}, $\sum_{l=1}^{n-1}P(A_l)=o(n^{1/2+\alpha})$ for all $\alpha>0$. Here we take $\alpha=1/6$ and then as $n\rightarrow\infty$, by (\ref{4.17}), (\ref{4.28}) and (\ref{4.30}),
 
 \begin{eqnarray}\label{4.23}
 \nonumber \sum_{l=1}^{n-1}\mathbb{E}\left[R_l^2\right] &=& n^{-1} \sum_{l=1}^{n-1}\sum_{i_1,i_2,i_3,i_4} \theta_{i_1}\theta_{i_2}\theta_{i_3}\theta_{i_4} \mathbb{E}\Big[\Big(E^{\omega}\big[I(\vec{X}_l)\big]-\mathbb{E}_{n,l-1}E^{\omega}\big[I(\vec{X}_l)\big] \Big)^2\Big] \\
 &\leq& C'\sum_{l=1}^{n-1}n^{-2}+C'n^{-2/3}\sum_{l=1}^{n-1}P(A_l)  =o(1)
 \end{eqnarray}

 This verifies condition (\ref{mcltcond1}), and therefore completes the proof of Lemma \ref{one-time-level}.
 \end{proof}
 \end{subsection}
 
 \begin{subsection}{Multiple time-levels}
 In this subsection we utilize the decomposition (\ref{5G}) to finish the second step stated at the beginning of this section, and thus finish the proof of Theorem \ref{GaussianlimitThm}.
 \begin{proof}[Proof of Theorem \ref{GaussianlimitThm}]
 In this step we use induction to show the convergence of the finite dimensional distributions.
 
 Assume that for some $M\in\mathbb{N}^+$, 
 \begin{equation}\label{4.24}
 \big(a_n(t_i,r_j):1\leq i\leq M, 1\leq j\leq N\big) \overset{d}{\Rightarrow}\big(a(t_i,r_j):1\leq i\leq M, 1\leq j\leq N\big)
 \end{equation}
 on $\mathbb{R}^{NM}$ for any finite $N$, $0\leq t_1<\cdots<t_M$ and $r_1<\cdots<r_N$.
 
 When $M=1$, it is just Lemma \ref{one-time-level}. It remains to show that (\ref{4.24}) also holds for $M+1$ time levels. Let $0\leq t_1<\cdots<t_M<t_{M+1}$. By the Cram\'er-Wold device, it suffices to prove that for any $(M+1)N$ vector $(\theta_{i,j})$,
 \begin{equation}
 \sum\limits_{1\leq i\leq M+1}\sum_{1\leq j\leq N}\theta_{i,j}a_n(t_i,r_j)\overset{d}{\Longrightarrow} \sum\limits_{1\leq i\leq M+1}\sum_{1\leq j\leq N}\theta_{i,j}a(t_i,r_j).
 \end{equation}
 
 For Borel sets $B\in\mathcal{B}$, let $\Delta:=t_{M+1}-t_M$, and denote the probability measures
 \begin{equation}
 p^{\omega}_{n,j}(B):=P^{\omega}\left(X^{r\sqrt{n}-\beta nt,-nt}_{-n(t-\Delta)}+\beta n\Delta\in B\right).
 \end{equation}
 Define
 $
 \tilde{a}_n(\Delta,r):=a_n(\Delta,r)\circ T_{-nt_M,-\beta nt_{M}}.
 $
 By the decomposition (\ref{5G}),
 \begin{equation}\label{5I}
 a_n(t_{M+1},r_j)=\int a_n(t_M,\frac{z}{\sqrt{n}})p^{\omega}_{n,j}(dz)+\tilde{a}_n(\Delta,r_j).
 \end{equation}
 In order to apply Lemma \ref{one-time-level}, we need to discretize the integral in (\ref{5I}). Given $A>0$, define a partition $\Pi$ of $[-A,A]$ by 
 \begin{equation}
 -A=u_0<u_1<\cdots<u_L=A.
 \end{equation}
 with mesh size $\delta=\max_{1\leq l\leq L}\{u_{l}-u_{l-1}\}$.Then
 \begin{eqnarray}
 a_n(t_{M+1},r_j)=\sum_{l=1}^{L}a_n(t_M,u_l)p_{n,j}^{\omega}\big((u_{l-1}\sqrt{n},u_l\sqrt{n}]\big)+\tilde{a}_n(s,r_j)+R_{n,j}(A),
 \end{eqnarray}
 where the error term $R_{n,j}(A)$ is given by
 \begin{eqnarray}
 \label{R_nj1}R_{n,j}(A)&=&\sum_{l=1}^{L}\int\limits_{(u_{l-1}\sqrt{n},u_l\sqrt{n}]} \left(a_n(t_M,\frac{z}{\sqrt{n}})- a_n(t_M,u_l)\right)p_{n,j}^{\omega}(dz) \\
 \label{R_nj2}&&+\left(\int\limits_{(-\infty,-A\sqrt{n}]}+\int\limits_{(A\sqrt{n},\infty)}\right) a_n(t_M,\frac{z}{\sqrt{n}})p_{n,j}^{\omega}(dz)
 \end{eqnarray}
 Let $R_n(A)=\sum_j\theta_{M+1,j}R_{n,j}(A)$, then we can rewrite
 \begin{align}\label{5J}
 \nonumber&~~~\sum\limits_{1\leq i\leq M+1}\sum_{1\leq j\leq N}\theta_{i,j}a_n(t_i,r_j) \\
 &=\sum\limits_{1\leq i\leq M}\sum_{1\leq k\leq K}\rho^{\omega}_{n,i,k}
 a_n(t_i,q_k)
 +\sum_{1\leq j\leq N}\theta_{M+1,j}\tilde{a}_n(s,r_j) + R_n(A).
 \end{align}
 In the above the spatial points $\{q_k\}$ are a relabeling of $\{r_j,u_l\}$, and the $\omega$-dependent coefficients $\rho^{\omega}_{n,i,k}$ consist of constants $\theta_{i,j}$, zeros and probabilities $p_{n,j}^{\omega}\big((u_{l-1}\sqrt{n},u_l\sqrt{n}]\big)$. By the quenched invariance principle Theorem \ref{qIP}, the constant limits $\rho^{\omega}_{n,i,k} \rightarrow\rho_{i,k}$ exist $\mathbb{P}$-a.s. as $n\rightarrow\infty$.
 
 To consider the limit $a(t,r)$, let $\tilde{a}(\Delta,\cdot)$ be a random function which is an independent copy of $a(\Delta,\cdot)$. By checking how the constants $\rho_{i,k}$ arise,
 \begin{eqnarray}\nonumber
 &&\sum\limits_{1\leq i\leq M}\sum\limits_{1\leq k\leq K}\rho_{i,k} a(t_i,q_k)+\sum\limits_{1\leq j\leq N}\theta_{M+1,j}\tilde{a} (\Delta,r_j)
 = \sum\limits_{1\leq i\leq M}\sum\limits_{1\leq j\leq N}\theta_{i,j} a(t_i,r_j) \\
 \label{4.33}&&+\sum_{1\leq j\leq N}\theta_{M+1,j}\Big(\sum_{l=1}^{l=L}\int_{u_{l-1}}^{u_l}a(t_M,r_j+u_l)\phi(z/\sqrt{\Delta})dz+\tilde{a}(\Delta,r_j)\Big)
 \end{eqnarray}
 
 Showing the weak convergence of the linear combination in (\ref{5J}) is equivalent to showing that for any bounded Lipschitz function $f$ on $\mathbb{R}$, (\ref{diff0}) below vanishes as $n$ tends to $\infty$. Note that
 \begin{eqnarray}
 \label{diff0}&&\mathbb{E}f\Bigg(\sum\limits_{1\leq i\leq M+1}\sum_{1\leq j\leq N}\theta_{i,j}a_n(t_i,r_j)\Bigg) - Ef\Bigg(\sum\limits_{1\leq i\leq M+1}\sum_{1\leq j\leq N}\theta_{i,j}a(t_i,r_j)\Bigg) \\
 \nonumber&=&~~\Bigg\{\mathbb{E}f\Bigg(\sum\limits_{1\leq i\leq M+1}\sum_{1\leq j\leq N}\theta_{i,j}a_n(t_i,r_j)\Bigg) \\
 \label{diff1}&&~~~~- \mathbb{E}f\Bigg(\sum\limits_{1\leq i\leq M}\sum_{1\leq k\leq K}\rho^{\omega}_{n,i,k} a_n(t_i,q_k)+\sum_{1\leq j\leq N} \theta_{M+1,j}\tilde{a}_n(\Delta,r_j)\Bigg)\Bigg\} \\
 \nonumber&&+ \Bigg\{\mathbb{E}f\Bigg(\sum\limits_{1\leq i\leq M}\sum_{1\leq k\leq K}\rho^{\omega}_{n,i,k} a_n(t_i,q_k)+\sum_{1\leq j\leq N} \theta_{M+1,j}\tilde{a}_n(\Delta,r_j)\Bigg) \\
 \label{diff2}&&~~~~ - Ef\Bigg(\sum\limits_{1\leq i\leq M}\sum\limits_{1\leq k\leq K}\rho_{i,k} a(t_i,q_k)+\sum\limits_{1\leq j\leq N}\theta_{M+1,j}\tilde{a} (\Delta,r_j)\Bigg)\Bigg\} \\
 \nonumber&&+\Bigg\{Ef\Bigg(\sum\limits_{1\leq i\leq M}\sum\limits_{1\leq k\leq K}\rho_{i,k} a(t_i,q_k)+\sum\limits_{1\leq j\leq N}\theta_{M+1,j}\tilde{a} (\Delta,r_j)\Bigg) \\
 \label{diff3}&&~~~~- Ef\Bigg(\sum\limits_{1\leq i\leq M+1}\sum_{1\leq j\leq N}\theta_{i,j}a(t_i,r_j)\Bigg)\Bigg\}.
 \end{eqnarray} 
 It remains to show that the three differences (\ref{diff1})-(\ref{diff3}) all converge to zero.
 
 By the Lipschitz continuity of $f$ and the decomposition (\ref{5J}), the difference (\ref{diff1}) is bounded by
 \begin{equation}
 C_f\mathbb{E}\left|R_n(A)\right|,
 \end{equation}
 where $C_f$ is the Lipschitz constant of $f$. To bound $R_n(A)$, it suffices to bound each $R_{n,j}(A)$, for which we will deal with the terms in (\ref{R_nj1}) and (\ref{R_nj2}) separately.
 First by Lemma \ref{variance}, the covariance of $(a_n(t,r))$ is given by
 \begin{equation}
 \mathbb{E}\left[a_n(t,r)a_n(t,q)\right]=n^{-1/2}\Big(G\big(\sqrt{n}(r-q),nt\big)+H\big(\sqrt{n}(r-q),nt\big)\Big),
 \end{equation}
 which, together with the fact that $\left|\frac{\partial G}{\partial x}(x,st)\right|\leq 2\nu$ and $H(x,t)\leq 6\nu^2$ uniformly in $(x,t)$ (see Lemma \ref{variance}), implies that
 \begin{eqnarray}
 \nonumber\mathbb{E}\left[\big(a_n(t,r)-a_n(t,q)\big)^2\right]&\leq& 2n^{-1/2}\Big(G\big(0,nt\big)-G\big(\sqrt{n}(r-q),nt\big)+12\nu\Big) \\
 &\leq& C_1\left|r-q\right|+C_2 n^{-1/2}. 
 \end{eqnarray}
 Using the independence of $a_n(t_M,r)$ and $p^{\omega}_{n,j}$ in (\ref{R_nj1}), the $L^1$ norm of (\ref{R_nj1}) can be bounded by
 \begin{eqnarray}
 \nonumber&&\mathbb{E}\left|\sum_{l=1}^{L}\int\limits_{(u_{l-1}\sqrt{n},u_l\sqrt{n}]} \left(a_n(t_M,\frac{z}{\sqrt{n}})- a_n(t_M,u_l)\right) p_{n,j}^{\omega}(dz) \right| \\
 \nonumber&\leq&\sum_{l=1}^{L}\left| \int\limits_{(u_{l-1}\sqrt{n}, u_l\sqrt{n}]}\left( \mathbb{E}\left[\left(a_n(t_M,\frac{z}{\sqrt{n}})- a_n(t_M,u_l)\right)^2\right]\right)^{1/2} \mathbb{E}\left[p_{n,j}^{\omega}(dz)\right] \right| \\
 &\leq& C'_1 \sqrt{\delta}+C'_2 n^{-1/4},
 \end{eqnarray}
 where $\delta$ is the mesh size of the partition $\Pi$.
 
 For the difference (\ref{R_nj2}), since $G(0,t)=\nu\sqrt{t}$ and 
 \begin{eqnarray}
 \nonumber\mathbb{E}\left[a^2_n(t,r)\right]= n^{-1/2}\Big(G\big(0,nt\big)+H\big(0,nt\big)\Big)
 \leq n^{-1/2}\big(\nu\sqrt{nt}+12\nu^2\big) = O(1),
 \end{eqnarray}
 we have
 \begin{eqnarray}
 \nonumber&&\mathbb{E}\left|\left(\int\limits_{(-\infty,-A\sqrt{n}]}+\int\limits_{(A\sqrt{n},\infty)}\right) a_n(t_M,\frac{z}{\sqrt{n}})p_{n,j}^{\omega}(dz)\right| \\
 \nonumber&\leq& \left(\int\limits_{(-\infty,-A\sqrt{n}]}+ \int\limits_{(A\sqrt{n},\infty)}\right) \left(\mathbb{E}\left[a^2_n(t_M,\frac{z}{\sqrt{n}})\right] \right)^{1/2} \mathbb{E}\left[p_{n,j}^{\omega}(dz)\right] \\
 &\leq& C P\left(\left|X_0^{z-\beta nt_M,-nt_M}\right|\geq A\sqrt{n}\right).
 \end{eqnarray}
 As a result, for any given $\epsilon>0$, we can first choose $A$ large enough and then $\triangle$ small enough so that the term (\ref{diff1}) satisfies
 \begin{equation}
 \limsup_{n\rightarrow\infty}[(\ref{diff1})] < \epsilon.
 \end{equation}
 
 To bound the difference (\ref{diff2}), we cite Lemma 5.3 in \cite{M.Bala.F.Rass.T.Sepp.2006}, which states as follows:
 \begin{lem}
 For any $k\in\mathbb{N}^+$, for each $n$, let $J_n=(J^1_n,\cdots,J^k_n)$, $X_n=(X^1_n,\cdots,X^k_n)$ and $Y_n$ be random variables in some probability space. If for each $n$, $X_n$ and $Y_n$ are independent, and marginally the weak convergences $J_n\Rightarrow j$, $X_n\Rightarrow X$ and $Y_n\Rightarrow Y$ hold, where $j$ is a constant $k$-vector, $X$ a random $k$-vector and $Y$ a random variable, then the weak convergence $J_n X_n+Y_n\Rightarrow jX+Y$ holds, where $X$ and $Y$ are independent.
 \end{lem}
 
 Now note that in (\ref{diff2}) $\rho^{\omega}_{n,i,k}\rightarrow\rho_{i,k}$ $\mathbb{P}$-a.s., hence in distribution. By the induction assumption (\ref{4.24}), $\{a_n(t_i,q_k):1\leq i\leq M,1\leq k\leq K \}$ converges weakly to $\{a(t_i,q_k):1\leq i\leq M,1\leq k\leq K \}$, and by Lemma \ref{one-time-level} $\{\tilde{a}_n(s,r_j):1\leq j\leq N\}$ converges weakly to $\{\tilde{a}(s,r_j):1\leq j\leq N\}$. Moreover, for each $n$, $a(t_i,q_k)$ is independent of $\tilde{a}(s,r_j)$. This implies
 \begin{equation}
 \lim_{n\rightarrow\infty}[(\ref{diff2})]=0.
 \end{equation}
 
 To bound difference (\ref{diff3}), the method is the same as for (\ref{diff1}). By Proposition \ref{prop1}, there is a representation (equal in finite dimensional distributions) of $a(t_{M+1},r_j)$ given by 
 \begin{equation}\label{4.45}
 a(t_{M+1},r_j):=\int a(t_M,r+z)\phi(z/\sqrt{\Delta})dz+\tilde{a}(\Delta,r_j).
 \end{equation}
 Substitute (\ref{4.33}) and (\ref{4.45}) into the first and second term of the difference (\ref{diff3}). Again since $f$ is Lipschitz, under the same partition $\Pi$ we have a similar error term as $R_n(A)$ in (\ref{5J}). Recall the covariance function $\Gamma((t,r),(t,q))= G(r-q,t)$ of $a(t,r)$. We also have
 \begin{eqnarray}
 E\left[\big(a(t,r)-a(t,q)\big)^2\right]\leq C_1|r-q|,
 \end{eqnarray}
 \begin{equation}
 E\left[a^2(t,r)\right]\leq C_2,
 \end{equation}
 which allows us to bound the error term with the same method as for (\ref{diff1}). Therefore, if we take large enough $A$ of the partition $\Pi$ and then make the mesh $\delta$ small enough, then
 \begin{equation}
 \limsup_{n\rightarrow\infty}[(\ref{diff3})]< \epsilon.
 \end{equation}
 
 In sum, given any bounded and Lipschitz continuous function $f$ and $\epsilon>0$, by choosing suitable partition $\Pi$,
 \begin{equation}
 \limsup_{n\rightarrow\infty}\left|\mathbb{E}f\Bigg(\sum\limits_{1\leq i\leq M+1}\sum_{1\leq j\leq N}\theta_{i,j}a_n(t_i,r_j)\Bigg) - Ef\Bigg(\sum\limits_{1\leq i\leq M+1}\sum_{1\leq j\leq N}\theta_{i,j}a(t_i,r_j)\Bigg)\right|<2\epsilon.
 \end{equation}
 This complete the proof of Theorem \ref{GaussianlimitThm}.
 \end{proof} 
 \end{subsection}
 
 \end{section}
 
 \begin{section}{Proof of Theorem \ref{final}}

 Before going into the proof, we do some analysis on $z_n(t,r)$ given in (\ref{DSPzn}). Denote $x(n,r):=nx_0+r\sqrt{n}$, then by the definition of $\zeta_t(x)$ in (\ref{DSP}),
 \begin{eqnarray}
 \label{4A}z_n(t,r) 
 &=& n^{-1/4} \left\{ \int f^{(n)}(y)K_{-nt,0}^{\omega} (x(n,r)-\beta nt,dy)-f^{(n)}(x(n,r)) \right\}\\
 \label{4B}&& +  n^{-1/4} \left\{ \int W(y)K_{-nt,0}^{\omega} (x(n,r)-\beta nt,dy)-W(x(n,r))\right\} \\
 \label{decomp.z_n}&=:& n^{-1/4} U_n(t,r)+n^{-1/4}V_n(t,r),
 \end{eqnarray}
 where $(K^{\omega}_{s,t})_{s\leq t}$ is the Howitt-Warren flow.
 We will consider the processes $U_n$ and $V_n$ separately. 
 
 To analyze $U_n(t,r)$, we break the domain of the integral in line (\ref{4A}) into two parts: $(-\infty,x(n,r))$ and $[x(n,r),\infty)$. Note that $f^{(n)}(x(n,r))=\int f^{(n)}\big(x(n,r)\big)K_{-nt,0}^{\omega} (x(n,r)-\beta nt,dy)$ and recall that $f^{(n)}(x)=nf(\frac{x}{n})$ in Assumption I. The integral over $[x(n,r),\infty)$ is equal to
 \begin{eqnarray}\label{5.4}
 \nonumber&&\int\limits_{[x(n,r),\infty)}\Big\{\int_{x(n,r)}^{y}f'(\frac{z}{n})dz \Big\} K_{-nt,0}^{\omega} (x(n,r)-\beta nt,dy) \\
 \nonumber&=& \int_{x(n,r)}^{\infty}f'(\frac{z}{n}) P^{\omega}\left(X^ {x(n,r)-\beta nt,-nt}_0>z\right) dz \\
 \nonumber&=& \int_{x(n,r)}^{\infty}f'(x_0) P^{\omega}\left(X^ {x(n,r)-\beta nt,-nt}_0>z\right) dz +R_n \\
 &=& f'(x_0) E^{\omega} \left[\left(X^{x(n,r)-\beta nt,-nt}_0- x(n,r)\right)^+\right] + R_n,
 \end{eqnarray}
 where $X^ {x(n,r)-\beta nt,-nt}$ is a random motion in the Howitt-Warren flow, and $R_n$ is the remainder dominated by
 \begin{eqnarray}\label{5.5}
 \nonumber R_n &=& \int_{x(n,r)}^{\infty}\bigg(f'(\frac{z}{n})-f'(x_0)\bigg) P^{\omega}\left(X^ {x(n,r)-\beta nt,-nt}_0>z\right) dz  \\
 \nonumber&\leq&\int_{x(n,r)}^{x(n,r)+n^{1/2+\delta}} \left|f'(\frac{z}{n})-f'(x_0)\right|dz  + C_0\int_{x(n,r)+n^{1/2+\delta}}^ {\infty} P^{\omega}\left(X^ {x(n,r)-\beta nt,-nt}_0>z\right)dz \\
 &=:& R_{n1}+R_{n2},
 \end{eqnarray}
 where in the inequality we bounded the quenched probability by 1 for the first term, and by Assumption I we take a bound $C_0/2$ of $f'$ for the second term. In (\ref{5.5}) $\delta$ can be any positive constant and will be chosen in (\ref{5.6}).
 
 For $R_{n1}$, by the H\"older continuity of $f'$ with H\"older constant $C$ and H\"older exponent $\gamma>1/2$,
 \begin{equation}\label{5.6}
 R_{n1}\leq Cn^{1/2+\delta} \left|\frac{r\sqrt{n}+n^{1/2+\delta}}{n}\right|^{\gamma}= o(n^{1/4}),
 \end{equation}
 where we only need to take $0<\delta<(2\gamma-1)/(4\gamma+4)$.
 
 As for $R_{n2}$, we are going to bound the term $\sup_{(t,r)\in A} R_{n2}(t,r)$, where $A:=[0,T]\times[-Q,Q]$ for some $T,Q>0$. We first couple $\{X^{x(n,r)-\beta nt,-nt}:(t,r)\in A\}$ in the Howitt-Warren flow $(K_{s,t})_{s\leq t}$ such that two random motions coalesce as soon as they meet. We then couple the backward random motions $\{\hat{X}^{x,0}:x\in\mathbb{R} \}$ in the dual Howitt-Warren flow $(\hat{K}_{t,s})_{t\geq s}$, as introduced in Section 1.3. By the non-crossing property between $(K_{s,t})_{s\leq t}$ and $(\hat{K}_{t,s})_{t\geq s}$,
 \begin{equation}\label{5.7}
 P^{\omega}\Big(\sup_{(t,r)\in A}X^ {x(n,r)-\beta nt,-nt}_0>z\Big)=\hat{P}^{\omega}\Big( \inf_{0\leq s\leq T}\big(\hat{X}^{z,0}_{-ns}-(nx_0+Q\sqrt{n}-\beta ns) \big)<0 \Big),
 \end{equation}
 where $\hat{P}^{\omega}$ denotes the quenched law for $\hat{X}^{z,0}$ in the dual flow $(\hat{K}_{t,s})_{t\geq s}$. Therefore, for any $\epsilon>0$,
 \begin{eqnarray} \label{6B}
 \nonumber\mathbb{P}\Big(\sup_{(t,r)\in A} R_{n2}(t,r)>\epsilon\Big)\leq\mathbb{P}\Big[C_0\int_{nx_0-Q\sqrt{n}+n^{1/2+\delta}}^ {\infty} P^{\omega}\Big(\sup_{(t,r)\in A}X^ {x(n,r)-\beta nt,-nt}_0>z\Big)dz>\epsilon\Big] \\
 \leq \epsilon^{-1} C_0 \int_{nx_0-Q\sqrt{n}+n^{1/2+\delta}}^ {\infty} \hat{P}\Big(\inf_{0\leq s\leq T}\big(\hat{X}^{z,0}_{-ns}-(nx_0+Q\sqrt{n}-\beta ns) \big)<0\Big)dz,
 \end{eqnarray}
 where in the first inequality we extended the lower integral bound to $nx_0-Q\sqrt{n}+n^{1/2+\delta}$ and then bounded the supremum by moving the supremum inside the quenched probability, while in the second inequality we applied (\ref{5.7}) and then the Markov inequality.
 Note that $(\hat{X}^{z,0}_t-\beta t)_{t\leq 0}$ is a backward Brownian motion under the dual averaged law $\hat{P}$, thus by the reflection principle, the probability in the right-hand side of (\ref{6B}) is $2\big[1-\Phi\big((z-nx_0-Q\sqrt{n})/T\big)\big]$, where $\Phi$ is the Gaussian distribution function. Note that the function $1-\Phi(x)$ has a Gaussian decay.
 Moreover, when $z$ is in the integral domain of (\ref{6B}), we have \begin{equation}
 z-nx_0-Q\sqrt{n}>n^{1/2+\delta}-2Qn^{1/2}=O(n^{1/2+\delta}).
 \end{equation}
 Therefore, the right-hand side of (\ref{6B}) is summable. By the Borel-Cantelli lemma, for $\mathbb{P}$ a.e. $\!\omega$,
 \begin{equation}
 \sup_{(t,r)\in A}R_{n2}(t,r)\longrightarrow 0,
 \end{equation}
 as $n\rightarrow\infty$.
 Applying the same argument to the integral in (\ref{4A}) on the domain $(-\infty,x(n,r))$, we can get a similar result.
 Hence the following lemma holds:
 \begin{lem}\label{lemma5.2}
 For any $T,Q>0$ denote $A=[0,T]\times[-Q,Q]$, then $\mathbb{P}$-a.s.,
 \begin{equation}\label{5.10}
 \lim\limits_{n\rightarrow\infty} \sup\limits_{(t,r)\in A} n^{-1/4} \Bigg|U_n(t,r)-f'(x_0)E^{\omega}\left[X^{x(n,r)-\beta nt,-nt}_0- x(n,r)\right]\Bigg|=0.
 \end{equation}
 \end{lem}
 
 As to the process $V_n(t,r)$, we have a representation:
 \begin{eqnarray}
 \nonumber &&\int W(y)K_{-nt,0}^{\omega} (x(n,r)-\beta nt,dy)-W(x(n,r)) \\
 \nonumber &=& \int_{x(n,r)}^{\infty}\Big\{\int_{x(n,r)}^{\infty}1_{(x(n,r),y)}(s)dW_s\Big\} P^{\omega}\left(X^{x(n,r)-\beta nt,-nt}_0 \in dy\right)\\
 \nonumber &&- \int^{x(n,r)}_{-\infty}\Big\{\int^{x(n,r)}_{-\infty}1_{(y,x(n,r))}(s)dW_s\Big\} P^{\omega}\left(X^{x(n,r)-\beta nt,-nt}_0 \in dy\right)\\
 \label{4E}&=& \int_{x(n,r)}^{\infty} P^{\omega}\left(X^{x(n,r)-\beta nt,-nt}_0 > s \right)dW_s\\
 \label{4F}&& -\int^{x(n,r)}_{-\infty} P^{\omega}\left(X^{x(n,r)-\beta nt,-nt}_0 < s \right)dW_s,
 \end{eqnarray}
 where in the last equality we interchanged the Lebesgue integral and It\^o integral because of the following lemma and the fact that $E^{\omega}|X^{x(n,r)-\beta nt,-nt}_0|$ exists for $\mathbb{P}$-a.e. $\omega$.
 \begin{lem}
 Let $X$ be a random variable on the probability space $(\Omega,\mathcal{F} ,P)$ with $E|X|<\infty$ and independent of the Brownian motion $W_t$. Then for any constant $c$, the following interchange of Lebesgue integral and It\^o integral is permissible:
 \begin{equation} \label{4C}
 \int_{c}^{\infty}\int_{c}^{\infty}1_{(c,y)}(s)dW_s P(X\in dy) =
 \int_{c}^{\infty}\int_{c}^{\infty}1_{(c,y)}(s)P(X\in dy) dW_s.
 \end{equation}
 \end{lem}
 \begin{proof}
 Without loss of the generality, we assume $c=0$.
 The left side of (\ref{4C}) can be approximated a.s. by 
 \begin{equation}\label{4D}
 \sum\limits_{k=0}^{\infty} P\left(k/2^n<X\leq (k+1)/2^n\right)\int_{0}^{\infty}1_{(0,k/2^n)} (s)dW_s,
 \end{equation}
 since a.s., Brownian motion is $\gamma$-H\"older continuous for all $\gamma<1/2$ and $P$ is a probability measure. On the other hand, since
 \begin{eqnarray}
 \nonumber&&\left|\int_{0}^{\infty}1_{(0,y)}(s)P(X\in dy)- \sum\limits_{k=0}^{\infty} P\left(\frac{k}{2^n}<X\leq \frac{k+1}{2^n}\right) 1_{(0,\frac{k}{2^n})} (s)\right|\\
 &\leq& \sum\limits_{k=0}^{\infty} P\left(\frac{k}{2^n}<X\leq \frac{k+1}{2^n}\right) 1_{(\frac{k}{2^n},\frac{k+1}{2^n})} (s),
 \end{eqnarray}
 the second moment of the difference of (\ref{4D}) and the right-hand side of (\ref{4C}) is bounded by
 \begin{eqnarray}
 \int_{0}^{\infty} \sum\limits_{k=0}^{\infty} P^2\left(\frac{k}{2^n}<X\leq \frac{k+1}{2^n}\right) 1_{(\frac{k}{2^n},\frac{k+1}{2^n})} (s) ds,
 \end{eqnarray}
 which is dominated by $\int_{0}^{\infty}P(X>s/2)ds\leq2E|X|<\infty$. Therefore, by dominated convergence theorem (\ref{4D}) converges to the right-hand side of (\ref{4C}) in $L^2$ as $n\rightarrow\infty$, and this shows that the interchange is permissible.
 \end{proof}
 With the form of (\ref{4E})-(\ref{4F}), it is clear that for $\mathbb{P}$-a.e.$~\omega$, $V_n(t,r)$ is a Gaussian process. Hence conditional on the random environment $\omega$ and integrating out the Brownian motion $W$, the conditional covariance of $V_n(t,r)$ is the same as the $L_2$ norm of the It\^o integral (\ref{4E})-(\ref{4F}) conditional on $\omega$, which is given by
 \begin{eqnarray} \label{5.17}
 \nonumber&& E^{\omega}\left[V_n(t,r)V_n(s,q) \right]\\ 
 \nonumber&=& \int_{x(n,r)\vee x(n,q)}^{\infty} P^{\omega}\left(X^{x(n,r) -\beta nt,-nt}_0 > z \right)P^{\omega}\left(X^{x(n,q)-\beta ns,-ns}_0 > z \right)dz \\
 \nonumber&&-1_{\{x(n,r)> x(n,q)\}} \int_{x(n,q)}^{x(n,r)} P^{\omega}\left(X^{x(n,r) -\beta nt,-nt}_0 < z \right) P^{\omega}\left(X^{x(n,q)-\beta ns,-ns}_0 > z \right)dz \\
 \nonumber&&-1_{\{x(n,r)< x(n,q)\}} \int^{x(n,q)}_{x(n,r)} P^{\omega}\left(X^{x(n,r) -\beta nt,-nt}_0 > z \right)P^{\omega}\left(X^{x(n,q)-\beta ns,-ns}_0 < z \right)dz \\
 \label{covIomega}&& +\int^{x(n,r)\wedge x(n,q)}_{-\infty} P^{\omega}\left(X^{x(n,r) -\beta nt,-nt}_0 < z \right)P^{\omega}\left(X^{x(n,q)-\beta ns,-ns}_0 < z \right)dz. 
 \end{eqnarray}
 By the quenched invariance principle Theorem \ref{qIP}, when $n$ tends to $\infty$, for $\mathbb{P}$-a.e. $\omega$,
 \begin{eqnarray}\label{5.18}
 P^{\omega}\big(X^{x(n,r) -\beta nt,-nt}_0 > nx_0+\sqrt{n}z \big)\longrightarrow P\big(W(t) > z-r \big), \\
 \label{5.19}
 P^{\omega}\big(X^{x(n,q) -\beta ns,-ns}_0 > nx_0+\sqrt{n}z \big)\longrightarrow P\left(W(s) > z-q \right),
 \end{eqnarray}
 where $W$ is a standard Brownian motion. Substituting (\ref{5.18})-(\ref{5.19}) into (\ref{5.17}) and noticing the expression of $\Gamma_0$ in (\ref{1.14}), we conclude that for $\mathbb{P}$-a.e. $\omega$,
 \begin{equation}\label{limitcovI_n}
 n^{-1/2}E^{\omega}\left[V_n(t,r)V_n(s,q) \right] \longrightarrow \Gamma_0((t,r),(s,q)) 
 \end{equation}
 as $n$ tends to $\infty$.
 With these useful observations, we utilize the following lemma cited from \cite[Lemma 7.1]{M.Bala.F.Rass.T.Sepp.2006} to show Theorem \ref{final}.
 
 \begin{lem}\label{lemma5.4}
 Let $(\Omega,\mathcal{F},\mathbb{P})$ and $(\Xi,\mathcal{G},\boldsymbol{P})$ be two probability spaces. On the product space $(\Omega\times\Xi, \mathcal{F}\times\mathcal{G}, \mathbb{P}\times \boldsymbol{P})$, define two sequences of $\mathbb{R}^N$-valued random vectors $\boldsymbol{U}_n(\omega)$ and $\boldsymbol{V}_n(\omega,\xi)$, where $\boldsymbol{U}_n$ depends only on $\omega$. Denote the conditional probability measure $P^{\omega}=\delta_{\omega}\times\boldsymbol{P}$ given $\omega$. If $\boldsymbol{U}_n$ and $\boldsymbol{V}_n$ satisfy the following two conditions:
 \begin{enumerate}
 \item[(i)] As $n$ tends to $\infty$, $\boldsymbol{U}_n$ converges weakly to an $\mathbb{R}^N$-valued random vector $\boldsymbol{U}$;
 \item[(ii)] There exists an $\mathbb{R}^N$-valued random vector $\boldsymbol{V}$ such that for all $\lambda\in\mathbb{R}$,
 \begin{equation}
 E^{\omega}\left[e^{i\lambda\cdot\boldsymbol{V}_n}\right]\longrightarrow E\left[e^{i\lambda\cdot\boldsymbol{V}}\right]
 \end{equation}
 in $\mathbb{P}$-probability as $n$ tends to $\infty$;
 \end{enumerate}
 then $\boldsymbol{U}_n+\boldsymbol{V}_n$ converges weakly to $\boldsymbol{U}+\boldsymbol{V}$, where $\boldsymbol{U}$ and $\boldsymbol{V}$ are independent.
 \end{lem}
 \begin{remark}{\rm In Lemma \ref{lemma5.4}, the limit of $\boldsymbol{U}_n+\boldsymbol{V}_n$ exists and consists of two independent parts. This is because for any $\omega$, there is a common limit of $\boldsymbol{V}_n$ which is not dependent on $\omega$, while $\boldsymbol{U}_n$ converges and only depends on $\omega$. The proof of the lemma is also straightforward, where we only need to show that the difference of $\mathbb{E}E^{\omega}\left[e^{i\theta\cdot\boldsymbol{U}_n+i\lambda\cdot\boldsymbol{V}_n}\right]$ and $E\left[e^{i\theta\cdot\boldsymbol{U}}\right]E\left[e^{i\lambda\cdot\boldsymbol{V}}\right]$ for arbitrary $\theta,\lambda\in\mathbb{R}^N$ vanishes when $n$ tends to $\infty$.}
 \end{remark} 
 
 Now we are ready to give the proof of Theorem \ref{final}.
 \begin{proof}[Proof of Theorem \ref{final}]
 We only need to show that for any $N$ space-time points $(t_1,r_1),\cdots,(t_N,r_N)$ in $\mathbb{R}^+\times\mathbb{R}$, $\left(z_n(t_1,r_1),\cdots,z_n(t_N,r_N)\right)$ converges weakly to $\left(z(t_1,r_1),\cdots,z(t_N,r_N)\right)$. According to the decomposition (\ref{decomp.z_n}), Lemma \ref{lemma5.2} and Theorem \ref{GaussianlimitThm}, $\left(n^{-1/4}U_n(t_1,r_1),\cdots,n^{-1/4}U_n(t_N,r_N)\right)$ depends only on $\omega$ and converges weakly to a random vector $\left(U(t_1,r_1),\cdots,U(t_N,r_N)\right)$. On the other hand, for $\mathbb{P}$-a.e. $\omega$, $\left(n^{-1/4}V_n(t_1,r_1),\cdots,n^{-1/4}V_n(t_N,r_N)\right)$ is Gaussian and thereby
 \begin{equation}
 E^{\omega}\left[e^{i\lambda\cdot n^{-1/4}V_n}\right] = \exp \{-\frac{1}{2}\lambda'\Sigma^{\omega}\lambda \},
 \end{equation}
 in which $\Sigma^{\omega}$ is the covariance matrix $\big(n^{-1/2}E^{\omega}\left[V_n(t_i,r_i)V_n(t_j,r_j) \right]\big)_{i,j}$ (see (\ref{covIomega}) for the expressions). By (\ref{limitcovI_n}), in the limit the matrix becomes $\Sigma=(\Gamma_0((t_i,r_i),(t_j,r_j))_{i,j}$, and therefore $\left(n^{-1/4}V_n(t_1,r_1),\cdots,n^{-1/4}V_n(t_N,r_N)\right)$ converges to a Gaussian vector $\left(V(t_1,r_1),\cdots,V(t_N,r_N)\right)$ satisfying condition (ii) in Lemma \ref{lemma5.4}. Hence there exists a mean zero Gaussian weak limit of $\left(z_n(t_1,r_1),\cdots,z_n(t_N,r_N)\right)$ in the form of Lemma \ref{lemma5.4}. From the covariances, the limit is $\left(z(t_1,r_1),\cdots,z(t_N,r_N)\right)$.
 \end{proof}
\end{section}

 \begin{appendices}
 \section{Quenched invariance principle for random walk in an i.i.d space-time random environment}
 The quenched invariance principle for one-dimensional random walk in an i.i.d space-time random environment has previously been established by Rassoul-Agha and Sepp\"al\"ainen in \cite{F.Rass.T.Sepp.2005}. They proved the result based on the view of the particle and martingale techniques. Now we apply the second moment method, similar to Section 3, to give an alternative proof of this result, since the method in this case is concise and self-contained. Actually, the second moment method was first used by Bolthausen and Sznitman in \cite{MR2023130}, and also by Comets and Yoshida in \cite{MR2271480}.
  
 In this basic model, an environment is a collection of transition probabilities $\omega=(\pi_{xy})_{x,y\in\mathbb{Z}}\in\Omega$ where $\Omega=\{(p_y)_{y\in\mathbb{Z}}\in[0,1]^{\mathbb{Z}}:\sum_{y}p_{y}=1\}^{\mathbb{Z}}$. The space $\Omega$ is equipped with the canonical product $\sigma$-algebra $\mathcal{F}$ and given an i.i.d probability measure $\mathbb{P}$. Here we say $\mathbb{P}$ is i.i.d in the sense that the distribution of the random probability vectors $(\pi_{xy})_{y\in\mathbb{Z}}$ are i.i.d over distinct sites $x$ and $\mathbb{P}$ is their product measure.
 
 Once the environment $\omega$ is chosen from the distribution $\mathbb{P}$, we fix it and sample a Markov process $X=(X_n)_{n\geq 0}$ with the state space $\mathbb{Z}$, starting from the site $z$, with the transition probability given by:
 \begin{eqnarray}
 \nonumber P^{\omega}_z(X_0=z)=1,
 \end{eqnarray}
 \begin{eqnarray}
 \nonumber P^{\omega}_z(X_{n+1}=y\big|X_n=x)=\pi_{xy}^{\omega}.
 \end{eqnarray}
 We call $X$ the one-dimensional random walk in an i.i.d space-time random environment. $P^{\omega}$ denotes the quenched law and we denote the averaged law by $P(\cdot):=\mathbb{E}P^{\omega}(\cdot)$. Under the averaged law, the averaged walk $X$ is just a random walk with transition probability $p(x,x+y)=p(0,y)=\mathbb{E}[\pi^{\omega}_{0y}]$. Let $\mu=\sum_{z}zp(0,z)$ and $\sigma^2=\sum_{z}z^2p(0,z)$ be the mean and the variance of the averaged walk. For $t\geq0$, define the linear interpolation of $X$ by $X_t:=X_{[n]}+(t-[t])(X_{[t]+1}-X_{[t]})$ where $[x]=\max\{n\in\mathbb{Z}:n\leq x\}$. Let $B_n(t)=\frac{X_{nt}-nt\mu}{\sqrt{n}\sigma}$ and $\tilde{B}_n(t)=\frac{X_{nt}-E^{\omega}[X_{nt}]}{\sqrt{n}\sigma}$ be random variables in $C[0,\infty)$, then  we have the following theorem:
 \begin{thm}
 With the notations introduced above, if $\sigma^2<\infty$ and $\mathbb{P}(\sup_{y\in\mathbb{Z}}\pi_{y}<1)>0$, then for $\mathbb{P}$-a.e. $\omega$, $B_n(t)$ converges weakly to a standard Brownian motion $B(t)$. Moreover, for $\mathbb{P}$-a.e. $\omega$, $n^{-1/2}\max_{k\leq n}\big|E^{\omega}X_{kt}-kt\mu\big|$ converges to $0$, and therefore the same a.s. invariance principle also holds for $\tilde{B}_n$.
 \end{thm}
 \begin{proof}
 The argument is similar to the one for Theorem \ref{qIP}. First we choose a proper coupling.
 Given the environment $\omega$, sample two independent random walk $(X^1_n)_{n\geq0}$ and $(X^2_n)_{n\geq0}$, and assume $X^1_0=X^2_0=0$ without loss of generality. Then under the averaged law $P$, $(X^1,X^2)$ is a Markov process with transition probability:
 \begin{eqnarray}
 \nonumber P\Big((X^1_{n+1},X^2_{n+1})=(y_1,y_2)\big|(X^1_{n},X^2_{n})=(x_1,x_2)\Big)=p(x_1,y_1)p(x_2,y_2), &~~~x_1\neq x_2; \\
 \nonumber P\Big((X^1_{n+1},X^2_{n+1})=(y_1,y_2)\big|(X^1_{n},X^2_{n})=(x,x)\Big)=\mathbb{E}[\pi^{\omega}_{xy_1}\pi^{\omega}_{xy_2}].&
 \end{eqnarray}
 Define a sequence of stopping times:
 \begin{eqnarray}
 \nonumber &\tau_0:=0; \\
 \nonumber &\tau_{n+1}:=\min\{k>\tau_{n}:X^1_k=X^2_k\}~~~(n\geq0).
 \end{eqnarray}
 and two sequences of i.i.d random variables in the same probability space $(\xi^1_n)_{n\geq0}$ and $(\xi^2_n)_{n\geq0}$, which are independent of $(X^1,X^2)$ and each other, with
 \begin{equation*}
 P(\xi^1_0=z)=P(\xi^2_0=z)=p(0,z)~~~(z\in\mathbb{Z}). 
 \end{equation*}
 Now couple a Markov process $(Y^1,Y^2)$ as follows:
 \begin{eqnarray}
 \nonumber (Y^1_0,Y^2_0)=(0,0),
 \end{eqnarray}
 for $\tau_k<n\leq\tau_{k+1}$,
 \begin{eqnarray}
 \nonumber  Y^j_n=X^j_n-X^j_{\tau_k+1}+\sum_{i=1}^{k}(X^j_{\tau_i}-X^j_{\tau_{i-1}+1}) + \sum_{i=0}^{k} \xi^j_{\tau_i} ~~~(j=1,2).
 \end{eqnarray}
 In words, the way $(Y^1_n,Y^2_n)$ jumping is exactly the same as that of $(X^1_n,X^2_n)$ when $X^1_n\neq X^2_n$. If $X^1_n= X^2_n=x_n$, then $Y^1_n$ and $Y^2_n$ still jump independently with transition probability $p(x_n,\cdot)$. Therefore, $Y^1$ and $Y^2$ are two independent random walks with transition kernels given by the function $p$. We also do linear interpolation to $(X^j_n)_{n\geq 0}$ and $(Y^j_n)_{n\geq 0}$ for $j=1,2$, and denote them by $(X^j_t)_{t\geq 0}$ and $(Y^j_t)_{t\geq 0}$.
 
 Now consider a bounded Lipschitz function $f:C[0,T]\rightarrow\mathbb{R}$ for any fixed $T>0$, where $C[0,T]$ is equipped with the sup norm. Denote $X^j_{k+1}-X^j_k$ by $\Delta X^j_k$, then
 \begin{eqnarray}
 \nonumber  E\Big[\sup_{0\leq t\leq T}|X^j_{nt}-Y^j_{nt}|\Big]  &\leq& \sum_{k=0}^{[nT]+1}E\left[1_{\{X^1_k=X^2_k\}}\big|\Delta X^j_k-\xi^j_k\big|\right] \\
 &\leq& 2\sigma\sum_{k=0}^{[nT]+1}P(X^1_k=X^2_k) 
 ~=~ O(n^{1/2}), 
 \end{eqnarray}
 where in the last inequality we applied the independence of $(X^1_k,X^2_k)$ and $(\Delta X^j_k,\xi^j_k)$, and the inequality $E\big|\Delta X^j_k-\xi^j_k\big|\leq\left(E\big[(\Delta X^j_k-\xi^j_k)^2\big]\right)^{1/2}\leq 2\sigma$; the last estimate is due to \cite[Lemma 3.3]{MR1624854}. Denote $\frac{X^j_{nt}-nt\mu}{\sqrt{n}\sigma}$ and $\frac{Y^j_{nt}-\mu}{\sqrt{n}\sigma}$ (j=1,2) by $B^j_n(t)$ and $W^j_n(t)$, then with the same argument as in the proof of Lemma \ref{3B} we have
 \begin{equation}
 \mathbb{E}\left[\left(E^{\omega}f(B^j_n)-Ef(W^j_n\right)^2\right]\leq Cn^{-1/4},
 \end{equation}
 Hence following the proof of theorem \ref{qIP} gives us that for $\mathbb{P}$-a.e. $\omega$, $E^{\omega}f(B_n)-Ef(W^1_n)\rightarrow 0$. Moreover, by Donsker's theorem $Ef(W^1_n(t))\rightarrow Ef(B(t))$, so for $\mathbb{P}$-a.e. $\omega$, $B_n(t)$ converges weakly to $B(t)$.
 
 As to the second part of the theorem, let $\sigma_0^2=\sum_{y_1,y_2\in\mathbb{Z}}(y_1-\mu)(y_2-\mu)\mathbb{E}[\pi_{0y_1}\pi_{0y_2}]$, then
 \begin{eqnarray}
 \nonumber\mathbb{E}\left[\left(E^{\omega}X_{[nt]}-[nt]\mu\right)^2\right]&=&
 E\left[\left(X^1_{[nt]}-[nt]\mu\right)\left(X^2_{[nt]}-[nt]\mu\right)\right]\\
 \nonumber &=& \sum_{k=0}^{[nt]-1}\sum_{l=0}^{[nt]-1}E\left[\left(X^1_{k+1}-X^1_{k}-\mu\right)\left(X^2_{l+1}-X^2_{l}-\mu\right)\right]\\
 \nonumber &=&\sigma^2_0\sum_{k=0}^{[nt]-1}P(X^1_k=X^2_k) 
 ~=~ O(n^{1/2}),
 \end{eqnarray}
 where in the second line the summands vanish unless $k=j$ and $X^1_k=X^2_l$ because $X^1_n$ and $X^2_n$ are independent when they do not meet. Again, with the same argument as in the proof of the Theorem \ref{qIP}, we have $n^{-1/2}\max_{k\leq n}\big|E^{\omega}X_{kt}-kt\mu\big|$ converges to $0$ for $\mathbb{P}$-a.e. $\omega$.
 \end{proof}

 \section{An integral identity}
 This section gives a probabilistic method to show the identity in Lemma \ref{variance} (iii).
 \begin{lem} For all $(t,x)\in\mathbb{R}^+\times\mathbb{R}$,
 \begin{equation}
 \int_{0}^{t}\frac{\sqrt{t-s}}{\pi s^{3/2}}|x|e^{-\frac{x^2}{2s}}ds=\int_{0}^{t}\frac{1}{\sqrt{2\pi s}}e^{-\frac{x^2}{2s}}ds.
 \end{equation}
 Therefore the statement of Lemma \ref{variance} (iii) holds.
 \end{lem}
 \begin{proof}
 To see the equality, we in turn consider the following probability question:
  For a standard Brownian motion $B_t$ starting from the point $x$, what is the expectation of the local time at the origin $\mathbb{E}[L_x(t,0)]$ up to time $t$. We compute this quantity in two ways.
  
  First we use the same method as in Section 2.2. Decompose it according to the first hitting time to $0$ of $B_t$. Since by Lemma \ref{lem-guji1} we have $\mathbb{E}[L_0(t,0)]=\sqrt{\frac{t}{2\pi}}$,
  \begin{eqnarray}\label{B2}
  \mathbb{E}[L_x(t,0)] = \int_{0}^{t} \mathbb{E}[L_0(t-s,0)] \frac{|x|}{\sqrt{2\pi}s^{3/2}}e^{-\frac{x^2}{2s}}ds= \int_{0}^{t}\frac{\sqrt{t-s}}{2\pi s^{3/2
   }}|x|e^{-\frac{x^2}{2s}}ds.
  \end{eqnarray}
  
  On the other hand, by Proposition \ref{localtime} (iv), we have
  \begin{equation}
  2\mathbb{E}\int_{-\infty}^{\infty}f(y)L_x(t,y)dy=\mathbb{E}\int_{0}^{t}f(B_s)ds=\int_{-\infty}^{\infty}\int_{0}^{t}f(y)\frac{1}{\sqrt{2\pi s}}e^{-\frac{(y-x)^2}{2s}}dsdy,
  \end{equation}
  for all measurable functions $f:\mathbb{R}\rightarrow [0,\infty)$. Therefore, 
  \begin{equation}\label{B4}
  \mathbb{E}[L_x(t,0)]=\int_{0}^{t}\frac{1}{2\sqrt{2\pi s}}e^{-\frac{x^2}{2s}}ds
  \end{equation}
  
  From (\ref{B2}) and (\ref{B4}) we get the result.
 \end{proof}
 
 \end{appendices}
 \section*{Acknowledgement}I am deeply indebted to my advisor Associate Professor Rongfeng Sun, who introduced me to this topic, and read carefully an earlier version of this paper and provided many corrections and suggestions. I also thank the referee who found many typos and provided useful suggestions.

\nocite{*}
\bibliographystyle{alpha}
\bibliography{reference}
\end{document}